\documentclass[12pt]{amsart}
\usepackage{a4wide,enumerate,xcolor,graphicx}
\usepackage{amsmath,comment}
\usepackage{scrextend} 
\allowdisplaybreaks

\usepackage{algorithm}
\usepackage{algorithmic}

\let\pa\partial
\let\na\nabla

\newcommand{\N}{{\mathbb N}}
\newcommand{\R}{{\mathbb R}}
\newcommand{\diver}{\operatorname{div}}

\newcommand{\E}{{\mathbb E}}
\renewcommand{\d}{\mathrm{d}}
\newcommand{\C}{{\mathcal C}}
\newcommand{\F}{{\mathcal F}}
\newcommand{\G}{{\mathcal G}}
\newcommand{\Prob}{{\mathbb P}}

\newtheorem{theorem}{Theorem}
\newtheorem{lemma}[theorem]{Lemma}

\newtheorem{proposition}[theorem]{Proposition}

\begin{document}

\title[Multi-species random-batch method]{Random-batch method for multi-species \\
stochastic interacting particle systems}

\author[E. S. Daus]{Esther S. Daus}
\address{Institute for Analysis and Scientific Computing, Vienna University of
	Technology, Wiedner Hauptstra\ss e 8--10, 1040 Wien, Austria}
\email{esther.daus@tuwien.ac.at}

\author[M. Fellner]{Markus Fellner}
\address{Institute for Analysis and Scientific Computing, Vienna University of
	Technology, Wiedner Hauptstra\ss e 8--10, 1040 Wien, Austria}
\email{markus.fellner@tuwien.ac.at}

\author[A. J\"ungel]{Ansgar J\"ungel}
\address{Institute for Analysis and Scientific Computing, Vienna University of
	Technology, Wiedner Hauptstra\ss e 8--10, 1040 Wien, Austria}
\email{juengel@tuwien.ac.at}

\date{\today}

\thanks{The authors have been partially supported by the Austrian Science Fund (FWF), 
grants P30000, P33010, F65, and W1245. This work received funding from the European 
Research Council (ERC) under the European Union's Horizon 2020 research and 
innovation programme, ERC Advanced Grant NEUROMORPH, no.~101018153.}

\begin{abstract}
A random-batch method for multi-species interacting particle systems is proposed,
extending the method of S.~Jin, L.~Li, and J.-G.~Liu
[{\em J. Comput. Phys.} 400 (2020), 108877]. The idea
of the algorithmus is to randomly divide, at each time step, 
the ensemble of particles into small batches
and then to evolve the interaction of each particle within the batches until
the next time step. This
reduces the computational cost by one order of magnitude, while keeping a certain
accuracy. It is proved that the $L^2$ error of the error process behaves 
like the square root
of the time step size, uniformly in time, thus providing the convergence of the scheme.
The numerical efficiency is tested for some examples, and
numerical simulations of the opinion dynamics in a hierarchical company, 
consisting of workers, managers, and CEOs, are presented.
\end{abstract}



\maketitle


\section{Introduction}

The collective behavior of particles or agents of multiple species 
can be described by interacting
particle systems, which are an important tool for modeling complex real-world
phenomena with applications in physics, biology, and social sciences. 
The binary interaction between all particles makes numerical simulations very demanding
when many agents need to be modeled,
which explains the need for efficient algorithms. Averaged results can be obtained
from the associated mean-field equations, 
while the individual dynamics is captured by direct simulations, 
using fast summation algorithms, like
fast multipole methods \cite{GrRo87}, wavelet transforms \cite{BCR91},
or variants of Monte--Carlo methods \cite{Caf98}. Recently, motivated by mini-batch 
gradient descent in machine learning (see, e.g., \cite{LZCS14}), 
the authors of \cite{JLL20} suggested to use small random batches
in interacting particle systems, 
which results in the reduction of the computational cost per time step 
from $O(N^2)$ to $O(N)$ ($N$ being the number of particles or agents).
Compared to other efficient sampling methods, like the Ewald summation or the
fast multipole method, the random-batch method is easier to implement and more
flexible to apply in complex systems. 
The results of \cite{JLL20} are valid in the single-species case.
In this paper, we generalize their approach to multi-species systems.
In particular, we work out the dependence of the $L^2$ error with respect to the
batch sizes of the different species and discuss the case of multiplicative noise.

\subsection{Setting}

The dynamics of the multi-species system is described by
\begin{align}
  \d X_i^k &= -\na V_i(X_i^k)\d t + \sum_{j=1}^n\alpha_{ij}\sum_{\substack{\ell=1 \\ 
	(i,k)\neq (j,\ell)}}^{N_j} K_{ij}(X_i^k-X_j^\ell)\d t 
	+ \sigma_i\d B_i^k(t), \label{1.eq} \\
	X_i^k(0) &= X_{0,i}^k \quad\mbox{for } i=1,\ldots,n,\ k=1,\ldots,N_i, \label{1.ic}
\end{align}
where
\begin{equation}\label{1.alpha}
  \alpha_{ij} = \frac{1}{N_j-\delta_{ij}}, \quad i,j=1,\ldots,n.
\end{equation}
The stochastic process $X_i^k(t)\in\R^d$ ($d\ge 1$) represents the position of
the $k$th particle (or the features of the $k$th agent)
of species $i$ in a system of $N=\sum_{i=1}^n N_i$ 
particles. The function $\na V_i$ describes some (given) external force, $K_{ii}$
and $K_{ij}$ are the interaction kernels between particles of the same and of
different species, respectively, $\sigma_i>0$ are diffusion coefficients, 
and $B_i^k$ are $N$ independent standard
Brownian motions. The initial data $X_{0,i}^1,\ldots,X_{0,i}^{N_i}$ are assumed 
to be independent and identically distributed. 

Equations \eqref{1.eq} can be used to model the information flow through social 
networks \cite{Ald13}, the dynamics of opinions \cite{FaRa21},
the herding of sheep by dogs \cite{SMWHMSK14},
or the segregation behavior of populations \cite{CDJ19}.
Stochastic gradient descent can be interpreted as the evolution of interacting
particle systems governed by a potential related to the objective function
used to train neural networks \cite{RoVa18}. 

\subsection{Random-batch method}

The random-batch method is defined as follows. 
Let the number of particles $N_i\in\N$ of the $i$th species
be an even number, where $i=1,\ldots,n$. We introduce the time steps
$t_m=m\tau$ with the time step size $\tau>0$ and $m=1,\ldots,M:=\lceil T/\tau\rceil$,
and $T>0$ is the end time. For a given $m\in\{1,\ldots,M\}$,
we divide the set $\{1,\ldots,N_i\}$ randomly into $b_i$ batches 
$\C_{i,1},\ldots,\C_{i,b_i}$ of size $p_i$. 
This means that we choose $p_i\ge 2$ and $b_i\ge 1$ such that $N_i=b_ip_i$,
and we consider not all interactions but only those in the same batch.
Furthermore, we introduce the super-batches $\C_r=\{(i,k):k\in\C_{i,r}\}$ for
$1\le r\le\max\{b_1,\ldots,b_n\}$ (see Figure \ref{fig.bat}). 
For any particle $X_i^k$, there exists exactly
one super-batch such that $(i,k)\in\C_r$ for some $r\ge 0$.

\begin{figure}[ht]
\includegraphics[width=85mm]{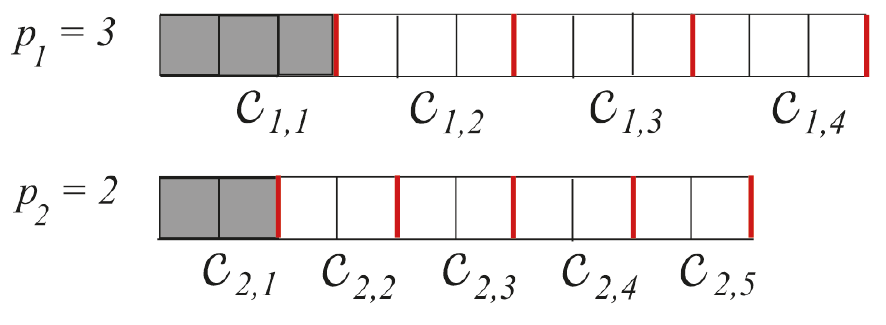}
\caption{Batches $\C_{i,r}$ for a two-species system with $N=22$ particles,
four batches of size
$p_1=3$, and five batches of size $p_2=2$. The particles in the super-batch $\C_1$
are marked in grey color.}
\label{fig.bat}
\end{figure}

We solve the particle system in the time interval
$(t_{m-1},t_m]$ with initial datum $\widetilde{X}_i^k(t_{m-1})$. 
The random-batch process $\widetilde{X}_i^k$ is defined 
for $t_{m-1}<t\le t_m$ as the solution to
\begin{equation}\label{1.rbm}
  \d\widetilde{X}_i^k = -\na V_i(\widetilde{X}_i^k)\d t + \sum_{j=1}^n\beta_{ij}
	\sum_{\substack{\ell\in\C_{j,r} \\ (i,k)\neq(j,\ell)}} K_{ij}
	(\widetilde{X}_i^k-\widetilde{X}_j^\ell)\d t + \sigma_i \d B_i^k, 
\end{equation}
where
\begin{equation}\label{1.beta}
  \beta_{ij} = \frac{b_i}{(p_j-\delta_{ij})\min\{b_i,b_j\}}, \quad i,j=1,\ldots,n.
\end{equation}
Instead of summing over all interactions, the sum in \eqref{1.rbm}
only accounts for the interactions in each small batch.
Observe that we use the same Brownian motions as in \eqref{1.eq}.
The sum over all $\ell\in\C_{j,r}$ means that we sum over all $(j,\ell)$
which are in the same super-batch as $(i,k)$. The factor $b_i/\min\{b_i,b_j\}$ 
in \eqref{1.beta} does not appear in \cite{JLL20}; it is necessary
to achieve consistency and convergence of the scheme. The scaling results from the
different number of nontrivial batches $\C_{i,r}$ of the different species.
Indeed, let $b_j<b_i$. From the viewpoint of the particles of the $i$th species, 
they interact with the particles of the $j$th species only with the share 
$b_j/b_i$, since $\C_{i,r}$ is empty for $r>b_j$. This yields the factor $b_i/b_j$. 
The random-batch algorithm is summarized in Algorithm \ref{algo}.

\begin{algorithm}
\label{algo}
\begin{algorithmic}[1]
\FOR{$k=m,\ldots,M$}
\FOR{$i=1,\ldots,n$}
\STATE Divide $\{1,\ldots,N_i\}$ randomly into $b_i$ batches 
$\C_{i,1},\ldots,\C_{i,b_i}$ with size $p_i$ each. 
\FOR{$r=1,\ldots,b_i$}
\STATE For every $(i,k)\in\C_r$, update $\widetilde{X}_i^k$ by solving
\eqref{1.rbm} in the interval $(t_{m-1},t_m]$ with initial datum 
$\widetilde{X}_i^k(t_{m-1})$.
\ENDFOR
\ENDFOR
\ENDFOR
\end{algorithmic}
\caption{(Pseudo-code for the multi-species random-batch algorithm)}
\end{algorithm}

When we allow for pairwise interactions between {\em all} particles, the computational
cost at each time step is of order $O(N^2)$. 
Since we have $M$ time steps, the total cost of this
naive algorithm is $O(MN^2)$. In the random-batch method, each particle ends up in
exactly one super-batch $\C_r$ for some $r\ge 1$
and is chosen only once (i.e.\ without replacement).
Then the total computatinal cost becomes $O(pMN)$, where $p=\sum_{i=1}^{n}p_{i}$. 
As $p$ is typically a small number (often $p_{i}=2$), the total cost has been 
reduced by approximately one order
of magnitude. We show in this paper that, under suitable conditions on the
external potentials and the kernel functions, the $L^2$ error
of the error process $\widetilde{X}_i^k(t)-X_i^k(t)$
converges to zero as $\tau\to 0$ uniformly in time, and the convergence
is, as expected, of order $O(\sqrt{\tau})$. The idea of the method is the fact
that in time average, the random force is consistent with the full interaction
(see Proposition \ref{prop.cons}), and the convergence is like in the law
of large numbers, but in time.

\subsection{Main result}

We start with some definitions and notation.
Let $(\Omega,\F,\mathbb{F},\Prob)$ be a filtered probability space, 
let $\xi_{m-1,i}$ denote the random division of batches of species $i$ 
at $t_{m-1}$, and set $\xi_{m-1}=(\xi_{m-1,1},\ldots,\xi_{m-1,n})$.
We define the filtrations $(\F_m)_{m\ge 0}$ and $(\G_m)_{m\ge 0}$ by
\begin{align*}
  \F_{m-1} &= \sigma\big(X_{0,i}^k,\,B_i^k(t),\,\xi_{j,i}:\,t\le t_{m-1},\,1\le i\le n,
	\,1\le j\le m-1\big), \\
	\G_{m-1} &= \sigma\big(X_{0,i}^k,\,B_i^k(t),\,\xi_{j,i}:\,t\le t_{m-1},\,1\le i\le n,
	\,1\le j\le m-2\big).
\end{align*}
The set $\F_{m-1}$ contains the information how the batches are constructed 
for $t\in[t_{k-1},t_k)$.
Denoting by $\sigma(\xi_{m-1,i})$ the $\sigma$-algebra generated by 
$\xi_{m-1,i}$, it holds that $\F_{m-1}=\sigma(\G_{m-1}\cup\sigma(\xi_{m-1,1})\cup
\cdots\cup\sigma(\xi_{m-1,n}))$. We write $\|\cdot\|_p
=(\E|\cdot|^p)^{1/p}$ to denote the $L^p(\Omega)$ norm for $1\le p<\infty$
and set $\|\cdot\|=\|\cdot\|_2$. In the whole paper,
$C>0$, $C_i>0$ denote generic constants whose values change from line to line.
We set $X=(X_i^k)_{i=1,\ldots,n}^{k=1,\ldots,N_i}$ and
$\widetilde{X}=(\widetilde{X}_i^k)_{i=1,\ldots,n}^{k=1,\ldots,N_i}$.

We impose the following assumptions:

\renewcommand{\labelenumi}{(A\theenumi)}
\begin{enumerate}
\item Kernel functions: $K_{ij}\in C^2(\R^d)$ is bounded, Lipschitz continuous
with Lipschitz constant $L_{ij}>0$, and has a bounded second derivative.
\item Potential functions: $V_i\in C^2(\R^d)$,
and there exist $C_V>0$, $q_i>0$ such that for all $x\in\R^d$,
$$
  |\na V_i(x)|+|D^2V_i(x)| \le C_V(1+|x|^{q_i}), \quad i=1,\ldots,n.
$$
\item Strong convexity: The function $x\mapsto V_i(x)
- r_i|x|^2/2$ is convex, where $r_i>2\sum_{j=1}^n$ $\max\{L_{ij},L_{ji}\}$ and
$i=1,\ldots,n$.
\item Synchronous coupling: $X_i^k(0)=\widetilde{X}_i^k(0)=X_{0,i}^k$ for
$i=1,\ldots,n$, $k=1,\ldots,N_i$, where $X_{0,i}^1,\ldots,X_{0,i}^{N_i}$ are 
independent and identically distributed, and 
$X_{0,i}^k$ is $\mathbb{F}_0$-measurable with
$\E|X_{0,i}^k|^{2\max\{1,q_i\}}<\infty$.
\end{enumerate}

Under these assumptions (in particular, the Lipschitz continuity), 
standard results for stochastic differential equations \cite{KlPl92}
guarantee that \eqref{1.eq} and \eqref{1.rbm} have 
(up to $\Prob$-distinguisha\-bi\-li\-ty)
unique strong solutions. The polynomial growth conditions on $\na V_i$ and
$D^2 V_i$ are needed to prove the stability; see Lemma \ref{lem.stab1}. 
The smallness condition on the Lipschitz constants
of the kernel functions ensures that the evolution group of the deterministic part
of \eqref{1.eq} is a contraction, thus yielding error bounds uniformly in time. 

Our main result reads as follows.

\begin{theorem}[Error estimate]\label{thm.main}
Let Assumptions (A1)--(A4) hold. Then there exists
a constant $C>0$, which is independent of $(b_i,p_i)_{i=1,\ldots,n}$,
$m$, and $T$, such that
$$
  \sup_{0<t<T}\sum_{i=1}^n\|(X_i^k-\widetilde{X}_i^k)(t)\|
	\le C\sqrt{\tau}\bigg(\sum_{i=1}^n\Gamma_i\bigg)^{1/2} 
	+ C\tau(1+\theta^\gamma), \quad t>0,
$$
where 
\begin{align}
  & \theta = \frac{\max_{j=1,\ldots,n}b_j}{\min_{j=1,\ldots,n}b_j}, \quad
	\gamma = 3\big(\max\{1,q_1,\ldots,q_n\}+1\big), \label{1.theta} \\
  \Gamma_i &= \sum_{\substack{j,j'=1 \\ j,j'\neq i,\, j\neq j'}}^n
	\bigg(\frac{\max\{b_{i},b_{j},b_{j'}\}}{\max\{b_{j},b_{j'}\}}-1 \bigg)
  + \sum_{\substack{j=1\\ j\neq i}}^{n}\bigg(\frac{b_{i}-\min\{b_{i},b_{j}\}}{
	\min\{b_{i},b_{j}\}} - \frac{2-\max\{b_{i},b_{j}\}}{N_{j}} \label{1.Gamma} \\
	&\phantom{xx}{}+ \frac{b_{i}}{p_{j}\min\{b_{i},b_{j}\}}
	\bigg)  + \bigg(\frac{1}{p_{i}-1}-\frac{1}{N_{i}-1}\bigg) \ge 0.
	\quad i=1,\ldots,n, \nonumber
\end{align}
and $q_i$ is introduced in Assumption (A2).
\end{theorem}

The theorem generalizes \cite[Theorem 3.1]{JLL20} to the multi-species case.
Indeed, if $n=1$, $\Gamma_1$ reduces to $1/(p_1-1)-1/(N_1-1)$ and $\theta=1$.
Then the error bound becomes $C\sqrt{\tau/(p_1-1)}+C\tau$, which corresponds
to (3.9) in \cite{JLL20}. Compared to the result in \cite{JLL20}, Theorem 
\ref{thm.main} shows the influence of the different batch sizes $b_i$ of the species.
Indeed, if the batch sizes are very different, $\theta$ is much larger than one,
which increases the constant in the error estimate. This behavior is also observed
in the numerical simulations; see Section \ref{sec.comp}.

The proof of Theorem \ref{thm.main} is based on estimates for the error process
$Z_i^k:=\widetilde{X}_i^k-X_i^k$. Since the noise terms are the same, $Z_i^k$
solves
$$
  \d Z_i^k(t) = -(\na V_i(\widetilde{X}_i^k)-\na V_i(X_i^k))\d t
	+ \sum_{j=1}^n\alpha_{ij}\sum_{\substack{\ell=1 \\ (j,k)\neq(j,\ell)}}
	\Delta K_{ij}^\ell\d t + \chi_i^k(\widetilde{X})\d t
$$
for $t_{m-1}<t\le t_m$, where $\Delta K_{ij}^\ell
:=K_{ij}(\widetilde{X}_i^k-\widetilde{X}_j^\ell)-K_{ij}(X_i^k-X_j^\ell)$
and $\chi_i^k(\widetilde{X})$ is a remainder term (defined in \eqref{2.chi} below).
An important ingredient of the proof is the computation of the variance of
$\chi_i^k$, which is more involved than in \cite{JLL20}, since the multi-species
case requires to distinguish several cases in the choice of indices $(i,k)$
and $(j,\ell)$.

A straightforward computation, detailed in Section \ref{sec.proof}, shows that
the error process satisfies
$$
  \frac12\frac{\d}{\d t}\E|Z_i^k(t)|^2 \le -\bigg(r_i - 2\sum_{j=1}^n
	\max\{L_{ij},L_{ji}\}\bigg)\E|Z_j^k(t)|^2 
	+ \E\big(\chi_i^k(\widetilde{X}(t))\cdot Z_i^k(t)\big).
$$ 
The main difficulty is the estimate of the last term. The idea is to write it
in terms of differences $Z_i^k(t)-Z_i^k(t_{m-1})$, 
$\chi_i^k(\widetilde{X}(t))-\chi_i^k(\widetilde{X}(t_{m-1}))$, and 
$\chi_i^k(\widetilde{X}(t))-\chi_i^k(X(t))$. 
These differences are estimated from the integral formulations of the differential
equations satisfied by the corresponding processes, using Assumptions (A1)--(A4)
and the stability results for $X_i^k$, $\widetilde{X}_i^k$, and $Z_i^k$.
After some computations, we arrive at the differential inequality
$$
  \frac{\d u}{\d t} \le -\min_{i=1,\ldots,n}
	\bigg(r_i-2\sum_{j=1}^n\max\{L_{ij},L_{ji}\}\bigg)u + C'(\theta)\tau(u^{1/2}+\tau)
	+ C''\tau\sum_{i=1}^n\Gamma_i,
$$
where $u=\sum_{i=1}^n\|Z_i^k\|^2$ and the constants $C'(\theta)>0$ and $C''>0$ 
do not depend on $(b_i,p_i)_{i=1,\ldots,n}$, $m$, or $T$. In view of Assumption (A2),
the first term on the right-hand side is nonpositive. 
The dependence of $C'(\theta)$ on $\theta$ arises from 
the terms involving $b_i/\min\{b_i,b_j\}$; see \eqref{1.beta}.
It follows that $u(t)$ is bounded from above by $C(\theta)\tau+C\sqrt{\tau
\sum_{i=1}^n\Gamma_i}$ for some other constants $C(\theta)>0$ and $C>0$.

\subsection{Link to related problems}

The random-batch scheme can be interpreted as a Monte--Carlo method to solve
the mean-field equations associated to \eqref{1.eq}. In the mean-field limit 
$N\to\infty$, system \eqref{1.eq} converges to 
$$
  \d\overline{X}_i = -\na V_i(\overline{X}_i)\d t
	+ \sum_{j=1}^n(K_{ij}* u_j)(\overline{X}_i)\d t + \sigma_i\d B_i, \quad
	i=1,\ldots,n,
$$
where $u_i$ is the probability density of $\overline{X}_i$ and solves the
mean-field system
$$
  \pa_t u_i = \diver(u_i\na V_i(x)) - \diver\bigg(\sum_{j=1}^n u_i(K_{ij}*u_j)\bigg)
	+ \frac{\sigma_i^2}{2}\Delta u_i\quad\mbox{in }\R^d,\,i=1,\ldots,n;
$$
see, e.g., the review \cite{JaWa17}. 
If $K_{ij}=\na k_{ij}$, it holds that $K_{ij}*u_j=k_{ij}*\na u_j$, and 
the density $u_i$ solves a nonlocal cross-diffusion system. Moreover, if
$k_{ij}=k_{ij}^\eta$ approximates the delta distribution $\delta$ according to
$k_{ij}^\eta\to a_{ij}\delta$ in ${\mathcal D}'$ as $\eta\to 0$
for some numbers $a_{ij}\ge 0$, it was shown in \cite{CDJ19} that
the limit $N\to\infty$ and $\eta\to 0$ (in a certain sense) 
leads to the local cross-diffusion system
$$
  \pa_t u_i = \diver(u_i\na V_i(x)) 
	- \diver\bigg(\sum_{j=1}^n a_{ij}u_i\na u_j\bigg)
	\quad\mbox{in }\R^d,\,i=1,\ldots,n.
$$
The mean-field limit of the random-batch method was investigated in
\cite{JiLi20}. The authors showed that the (single-species) $N$-particle system
is reduced to a $p$-particle system. This mean-field limit does not depend on the
law of large numbers, and it is different from the standard mean-field limit,
since the chaos is imposed at every time step, while in the standard limit, the
chaos is propagated to later times.

The idea of choosing particles in a random way has been exploited in kinetic
theory. For instance, subsampling was used in Monte--Carlo simulations \cite{HaTe18}
and for the symmetric Nabu algorithm, which relates to the random-batch method
for $p_i=2$ \cite{AlPa13}.

Random-batch methods can also be applied to second-order particle systems
\cite{JLS20}, many-particle Schr\"odinger equations \cite{GJP21}, and kinetic
equations \cite{LLT20}. They have been used to sample complicated or unknown
probability distributions \cite{LLL20,YeZh21}, and 
they have been combined with model predictive control
strategies to control the guiding problem for a herd of evaders \cite{KoZu20}.
In molecular dynamics, the interaction kernel is generally singular and given by,
e.g., the Coulomb or Lennard--Jones potential. This situation is excluded
in this paper because of Assumption (A1). 
However, one may split the kernel function into
(singular) short-range and (smooth) long-range parts and apply the random-batch
method only to the long-range part. This yields similar convergence results as
above but with constants depending on the end time \cite{JLS20}.
We refer to the review \cite{JiLi21} for further applications and references.

Theorem \ref{thm.main} provides the strong convergence with rate $O(\sqrt{\tau})$
of the error process. In \cite{JLL21}, 
the weak convergence with rate $O(\tau)$ is proved for the
single-species case. The proof makes use of the backward Kolmogorov equation
and the contraction of the associated semigroup in $L^\infty(\R^d)$.   
In the multi-species situation, we obtain a system of equations for which
contraction properties can be expected under Assumption (A2), but possibly in
a weaker topology. A possible wayout is to use estimates in the space
$H^s(\R^d)\subset L^\infty(\R^d)$ for $s>d/2$, derived for the mean-field limit 
\cite{CDJ19}. We leave the details to future work.

Theorem \ref{thm.main} can be generalized to particle systems with multiplicative
noise when the diffusion coefficients are Lipschitz continuous.
We can only prove stability for particle systems with {\em interacting} 
diffusion coefficients like in \cite{CDHJ20}, 
which lead in a mean-field-type limit to the
Shigesada--Kawasaki--Teramoto population model.
For details, we refer to Section \ref{sec.multi}.

The paper is organized as follows. The consistency of the scheme and stability of the 
stochastic processes $X_i^k$ and $\widetilde{X}_i^k$ are proved in 
Section \ref{sec.cons}. Section \ref{sec.est} is concerned with the control of the
error process $Z_i^k=\widetilde{X}_i^k-X_i^k$ and corresponding uniform estimates.
Theorem \ref{thm.main} is proved in Section \ref{sec.proof}. 
We comment on the error estimate for particle systems with multiplicative noise
in Section \ref{sec.multi}. Some numerical
simulations, illustrating the convergence behavior and the influence of the batch
sizes, are presented in Section \ref{sec.num}. Finally, we collect some
known results about the conditional expectation used in this paper in 
Appendix \ref{app}.


\section{Consistency and stability}\label{sec.cons}

We assume that Assumptions (A1)--(A4) hold.
Let $i\in\{1,\ldots,n\}$, $k\in\{1,\ldots,N_i\}$ and 
let $\widetilde{X}_i^k$ with $k\in\C_{i,r}$ be a solution to \eqref{1.rbm}.
Then $\widetilde{X}_i^k$ solves
\begin{equation}\label{1.rbm2}
  \d\widetilde{X}_i^k = -\na V_i(\widetilde{X}_i^k)\d t + \sum_{j=1}^n\alpha_{ij}
	\sum_{\substack{\ell=1 \\ (i,k)\neq(j,\ell)}}^{N_j}K_{ij}(\widetilde{X}_i^k
	- \widetilde{X}_j^\ell)\d t + \sigma_i\d B_i^k + \chi_{i}^k(\widetilde{X})\d t,
\end{equation}
where the remainder $\chi_i^k$ is defined for $x=(x_1^1,\ldots,x_n^{N_n})
\in\R^{dN_1\times\cdots\times dN_n}$ by
\begin{align}
  \chi_i^k(x) &= \sum_{j=1}^n\beta_{ij}\sum_{\substack{\ell\in\C_{j,r} \\ 
	(i,k)\neq(j,\ell)}} K_{ij}(x_i^k-x_j^\ell)
	- \sum_{j=1}^n\alpha_{ij}\sum_{\substack{\ell=1 \\ 
	(i,k)\neq (j,\ell)}}^{N_j} K_{ij}(x_i^k-x_j^\ell) \label{2.chi} \\
	&=: f_i^k(x) - g_i^k(x). \nonumber
\end{align}
The following proposition shows that the scheme is consistent.

\begin{proposition}[Consistency]\label{prop.cons}
Let $p_i\ge 2$ for $i=1,\ldots,n$ and $x=(x_1^1,\ldots,x_n^{N_n})
\in\R^{dN_1\times\cdots\times dN_n}$.
Then the expectation and variance of $\chi_i^k$, defined in \eqref{2.chi}, are
$\E(\chi_i^k(x))=0$ and
\begin{align}
  &\operatorname{Var}(\chi_i^k(x)) = \sum_{\substack{j,j'=1 \\ j,j'\neq i,\,j\neq j'}}^n
	\bigg(\frac{\max\{b_i,b_j,b_{j'}\}}{\max\{b_i,b_{j'}\}}-1\bigg)A_i^{jj'}(x) 
	\nonumber \\
	&\phantom{x}{}
	+ \sum_{\substack{j=1 \\ j\neq i}}^n\bigg(\frac{b_i-\min\{b_i,b_j\}}{\min\{b_i,b_j\}}
	- \frac{1}{N_j} + \frac{b_i}{p_j\min\{b_i,b_j\}}\bigg)A_i^j(x)
	+ \frac{\max\{b_{i},b_{j}\}-1}{N_{j}}A_{i,1}^j(x) \label{2.var} \\
	&\phantom{x}{}+ \bigg(\frac{1}{p_i-1}-\frac{1}{N_i-1}\bigg)A_i(x),\nonumber 
\end{align}
where $i=1,\ldots,n$, $k=1,\ldots,N_i$, and
\begin{align*}
	A_i^{jj'}(x) &= \frac{1}{N_jN_{j'}}\sum_{\ell=1}^{N_j}\sum_{\ell'=1}^{N_{j'}}
	K_{ij}(x_i^k-x_j^\ell)K_{ij'}(x_i^k-x_{j'}^{\ell'}), \\
	A_i^j(x) &= \frac{1}{N_{j}(N_j-1)}\sum_{\ell,\ell'=1,\,\ell\neq \ell'}^{N_j}
	K_{ij}(x_i^k-x_j^\ell)K_{ij}(x_i^k-x_{j}^{\ell'}), \\
	A_{i,1}^{j}(x) &= \frac{1}{N_{j}}\sum_{\ell=1}^{N_j}K_{ij}(x_i^k-x_j^\ell)^2, \\
	A_i(x) &= \frac{1}{N_i-2}\sum_{\ell=1,\, \ell\neq k}^{N_i}
	\bigg(K_{ii}(x_i^k-x_i^\ell) - \frac{1}{N_i-1}\sum_{\ell'=1,\,\ell'\neq k}^{N_i}
	K_{ii}(x_i^k-x_i^{\ell'})\bigg)^2.
\end{align*}
\end{proposition}

Using definition \eqref{1.Gamma}, we can estimate the variance of $\chi_i^k(x)$
from above according to
$$
  \operatorname{Var}(\chi_i^k(x))\le 8\max_{i,j=1,\ldots,n}\|K_{ij}\|_\infty^{2}
  \sum_{k=1}^n\Gamma_k.
$$
As expected, for larger batch sizes $p_i$, the variance is smaller 
and the noise level is lower.
In the single-species case, we recover \cite[Lemma 3.1]{JLL20} since 
$$
  \operatorname{Var}(\chi_i^k(x)) = \bigg(\frac{1}{p_i-1}-\frac{1}{N_i-1}\bigg)A_i(x).
$$ 
If the species numbers and batch sizes are the same, i.e.\ $N_i=N$ and $b_i=b$
for all $i=1,\ldots,n$, it follows that
$$
  \operatorname{Var}(\chi_i^k(x)) = \bigg(\frac{1}{p}-\frac{1}{N}\bigg)
	\sum_{j=1,\, j\neq i}^n \bigg(A_i^j(x) + A_{i,1}^j(x)\bigg) 
	+ \bigg(\frac{1}{p-1}-\frac{1}{N-1}\bigg)A_i(x).
$$
We observe that the first term on the right-hand side of \eqref{2.var} vanishes. 
This means that, in case of different species numbers or batch sizes, 
the noise level is larger than in the uniform case.

\begin{proof}
The proof is similar to \cite[Lemma 3.1]{JLL20}, but since we have multiple species,
the computations are more involved. Let $i\in\{1,\ldots,n\}$ and $k\in\{1,\ldots,N_i\}$
be arbitrary but fixed.
We write $I_i^k(j,\ell)=1$ if $(i,k)$ and $(j,\ell)$ are in the same batch,
i.e., if there exists $r\ge 1$ such that $(i,k)$, $(j,\ell)\in\C_r$.
Otherwise, we set $I_i^k(j,\ell)=0$. With this notation, we can write 
$f_i^k=f_i^k(x)$, defined in \eqref{2.chi}, as
$$
  f_i^k = \sum_{j=1}^n\beta_{ij}\sum_{\ell=1}^{N_j}K_{ij}(x_i^k-x_j^\ell)
	I_i^k(j,\ell).
$$

{\em Step 1: Computation of the expection.} We claim that
\begin{equation}\label{2.EI}
  \E I_i^k(j,\ell) = \left\{\begin{array}{ll}
	\frac{p_i-1}{N_i-1} &\mbox{if }i=j, \\
	\frac{\min\{b_i,b_j\}}{b_i b_j} &\mbox{if }i\neq j.
	\end{array}\right.
\end{equation}
The case $i=j$ is proved in \cite[Lemma 3.1]{JLL20}. For $i\neq j$, we define
$a(i,k)$ as the index of the super-batch $\C_r$ that contains $(i,k)$, i.e.\
$a(i,k)=r$ if and only if $(i,k)\in\C_r$ or, equivalently, $k\in\C_{i,r}$. We have
\begin{align*}
  \Prob(I_i^k(j,\ell)=1) &= \Prob((j,\ell)\in\C_{a(i,k)})
	= \sum_{r=1}^{\min\{b_i,b_j\}}\Prob((j,\ell)\in\C_r|a(i,k)=r)
	\Prob(a(i,k)=r) \\
	&= \sum_{r=1}^{\min\{b_i,b_j\}}\Prob((j,\ell)\in\C_r)\Prob(a(i,k)=r).
\end{align*}
The distribution of a particle of a certain species is uniform with respect to
the species' batch in which it ends up, i.e.\
$\Prob(\ell\in\C_{j,r})=\Prob(\ell\in\C_{j,s})$ for all $r,s=1,\ldots,b_j$.
Consequently, we have $\Prob(\ell\in\C_{j,r})=1/b_j$ for all $r=1,\ldots,b_j$ and
$\Prob(\ell\in\C_{j,r})=0$ otherwise, since $\C_{j,r}=\emptyset$ if $r>b_j$.
This leads for $i\neq j$ to
$$
  \E I_i^k(j,\ell) = 1\cdot\Prob(I_i^k(j,\ell)=1) 
	= \min\{b_i,b_j\}\frac{1}{b_j}\,\frac{1}{b_i}.
$$
We infer from the definitions of $\alpha_{ij}$ and $\beta_{ij}$ in \eqref{1.alpha}
and \eqref{1.beta}, respectively, and from \eqref{2.EI} that 
\begin{align*}
  \E(f_i^k) &= \sum_{j=1}^n\beta_{ij}\sum_{\ell=1}^{N_j}K_{ij}(x_i^k-x_j^\ell)
	\E I_i^k(j,\ell) \\
	&= \frac{1}{N_i-1}\sum_{\ell=1,\,\ell\neq k}^{N_i}K_{ii}(x_i^k-x_i^\ell)
	+ \sum_{j=1,\, j\neq i}^n\frac{1}{N_j}\sum_{\ell=1}^{N_j}K_{ij}(x_i^k-x_j^\ell) \\
	&= \sum_{j=1}^n\alpha_{ij}\sum_{\substack{\ell=1 \\ (i,k)\neq(j,\ell)}}
	K_{ij}(x_i^k-x_j^\ell) = \E(g_i^k).
\end{align*}
This shows that $\E(\chi_i^k(x))=\E(f_i^k)-\E(g_i^k)=0$.

{\em Step 2: Preparation for the computation of the variance.}
We introduce the notation $G_j^\ell:=K_{ij}(x_i^k-x_j^\ell)I_i^k(j,\ell)$
if $(i,k)\neq(j,\ell)$ and $G_j^\ell=0$ if $(i,k)=(j,\ell)$. Then
$$
  \E(f_i^k)^2 = \sum_{j,j'=1}^n\beta_{ij}\beta_{ij'}\sum_{\ell=1}^{N_j}
	\sum_{\ell'=1}^{N_{j'}}\E(G_j^\ell G_{j'}^{\ell'}).
$$
The expectation of $G_j^\ell G_{j'}^{\ell'}$ can be written as 
\begin{align*}
  \E(G_j^\ell G_{j'}^{\ell'}) &= K_{ij}(x_i^k-x_j^\ell)K_{ij'}(x_i^k-x_{j'}^{\ell'})
	\E(I_i^k(j,\ell)I_i^k(j',\ell')) \\
	&= K_{ij}(x_i^k-x_j^\ell)K_{ij'}(x_i^k-x_{j'}^{\ell'})
	\Prob(I_i^k(j,\ell)I_i^k(j',\ell')=1).
\end{align*}
Thus, we need to calculate $\Prob(I_i^k(j,\ell)I_i^k(j',\ell')=1)$. For this, 
we distinguish several cases.

{\em Case 1: $j,j'\neq i$ and $j\neq j'$.} 
We compute, using the definition of the super-batches,
\begin{align*}
  \{I_i^k&(j,\ell)I_i^k(j',\ell')=1\}
	= \big\{(j,\ell)\in\C_{a(i,k)},\,(j',\ell')\in\C_{a(i,k)}\big\} \\
	&= \big\{\ell\in\C_{j,a(i,k)},\,\ell'\in\C_{j',a(i,k)}\big\}
	= \bigcup_{r\in\N}\big\{\ell\in\C_{j,a(i,k)},\,\ell'\in\C_{j',a(i,k)},\,
	r = a(i,k)\big\}.
\end{align*}
The random division $\xi_{m.1}$ of the batch $\C_{i,r}$ at time $t_{m-1}$
is independent of the random devision of the batches $\C_{j,r}$
and $\C_{j',r}$. Thus, we can write
\begin{align*}
  \Prob(I_i^k(j,\ell)I_i^k(j',\ell')=1) &= \sum_{r=1}^{\min\{b_i,b_j,b_{j'}\}}
	\Prob(k\in\C_{i,r})\Prob(\ell\in\C_{j,r})\Prob(\ell'\in\C_{j',r}) \\
	&= \min\{b_i,b_j,b_{j'}\}\frac{1}{b_i}\,\frac{1}{b_j}\,\frac{1}{b_{j'}}.
\end{align*}

{\em Case 2: $j,j'\neq i$, $j=j'$ and $\ell\neq \ell'$.} In this case, 
both $\ell$ and $\ell'$ are in the same batch such that
\begin{align}
  \{I_i^k(j,\ell)I_i^k(j,\ell')=1\} &= \bigcup_{r\in\N}\big\{\ell,\ell'\in\C_{j,r},\,
	r = a(i,k)\big\} \label{2.aux1} \\
	&= \bigcup_{r\in\N}\{I_j^k(j,\ell')=1\}\cap\{\ell\in\C_{j,r}\}\cap\{k\in\C_{i,r}\}.
	\nonumber
\end{align}
Because of the uniformity of the random devision (as in Case 1), we have 
$$
  \Prob(I_j^\ell(j,\ell')=1,\,\ell\in\C_{j,r})
	= \Prob(I_j^\ell(j,\ell')=1,\,\ell\in\C_{j,s})\quad\mbox{for all }1\le r,s\le b_j.
$$
Since, by \eqref{2.EI}, $\Prob(I_j^\ell(j,\ell')=1) = (p_j-1)/(N_j-1)$,
we deduce from \eqref{2.aux1} that
\begin{align*}
  \Prob(I_i^k(j,\ell)I_i^k(j',\ell')=1) &= \sum_{r=1}^{\min\{b_i,b_j\}}
	\Prob(I_j^\ell(j,\ell')=1,\,\ell\in\C_{j,r})\Prob(k\in\C_{i,r}) \\
	&= \frac{1}{b_ib_j}\sum_{r=1}^{\min\{b_i,b_j\}}\Prob(I_j^\ell(j,\ell')=1)
	= \frac{\min\{b_i,b_j\}(p_j-1)}{b_ib_j(N_j-1)}. 
\end{align*}

{\em Case 3: $j\neq i$, $j'=i$.} If $\ell=k$, it follows from the definition of
$G_j^\ell$ that $G_i^{\ell'}=G_i^k=0$. If $\ell\neq k$, the definition of $G_j^\ell$
gives
$$
  \E(G_j^\ell G_{j'}^{\ell'}) = K_{ij}(x_i^k-x_j^\ell)K_{ii}(x_i^k-x_i^{\ell'})
	\E(I_i^k(j,\ell)I_i^k(i,\ell')),
$$
and it remains the compute the expectation on the right-hand side. 
Proceeding as in the previous cases, we find that
\begin{align*}
  \Prob(I_i^k(j,\ell)I_i^k(i,\ell')=1) &= \sum_{r=1}^{\min\{b_i,b_j\}}
	\Prob(I_i^k(i,\ell')=1,\,k\in \C_{i,r})\Prob(\ell\in\C_{j,r}) \\
	&= \frac{1}{b_ib_j}\sum_{r=1}^{\min\{b_i,b_j\}}\Prob(I_i^k(i,\ell')=1)
	= \frac{\min\{b_i,b_j\}(p_i-1)}{b_ib_j(N_i-1)}.
\end{align*}

{\em Case 4: $j,j'=i$, $\ell\neq\ell'\neq k$.} We need to compute the probability
of $I_i^k(i,\ell)I_i^k(i,\ell')=1$. This case happens exactly when the indices
$\ell$, $\ell'$, and $k$ are in the same batch $\C_{i,a(i,k)}$. Similar arguments
as for $\Prob(I_i^k(i,\ell)=1)$ in the proof of Lemma 3.1 in \cite{JLL20} yield
$$
  \Prob(I_i^k(i,\ell)I_i^k(i,\ell')=1) = \frac{(p_i-1)(p_i-2)}{(N_i-1)(N_i-2)}.
$$

{\em Case 5: $j,j'=i$, $\ell=\ell'$, $\ell\neq k$.} We only need
$\E(I_i^k(i,\ell))$, which we already computed:
$$
  \Prob(I_i^k(i,\ell)=1) = \frac{p_i-1}{N_i-1}.
$$

Summarizing these five cases, we obtain $\E(f_i^k)^2=J_1+\cdots+J_5$, where
the term $J_j$ corresponds to case $j$ and
\begin{align*}
  J_1 &= \sum_{\substack{j,j'=1\\ j,j'\neq i,\,j\neq j'}}^{n}
	\frac{\min\{b_{i},b_{j},b_{j'}\} b_{i}}{N_{j}N_{j'}\min\{b_{i},b_{j}\}
	\min\{b_{i},b_{j'}\}}\sum_{\ell,\ell'=1}^{N_{j},N_{j'}}
	K_{ij}(x_{i}^{k}-x_{j}^{\ell})K_{ij'}(x_{i}^{k}-x_{j'}^{\ell'}), \\
	J_2 &= \sum_{\substack{j=1\\ j\neq i}}^{n}\frac{(p_{j}-1)b_{i}}{(N_{j}-1)N_{j}
	\min\{b_{i},b_{j}\}p_{j}}\sum_{\ell,\ell'=1,\, \ell\neq \ell'}^{N_{j}}
	K_{ij}(x_{i}^{k}-x_{j}^{\ell})K_{ij}(x_{i}^{k}-x_{j}^{\ell'}), \\
	J_2' &= \sum_{\substack{j=1\\ j\neq i}}^{n}\frac{b_{i}}{N_{j}
	\min\{b_{i},b_{j}\}p_{j}}\sum_{\ell=1}^{N_{j}}K_{ij}(x_{i}^{k}-x_{j}^{\ell})^2, \\
	J_3 &= 2\sum_{\substack{j=1\\ j\neq i}}^{n}\frac{1}{(N_{i}-1)N_{j}}
	\sum_{\substack{\ell,\ell'=1\\ \ell'\neq k}}^{N_{j},N_{i}}
	K_{ij}(x_{i}^{k}-x_{j}^{\ell})K_{ii}(x_{i}^{k}-x_{i}^{\ell'}), \\
	J_4 &= \frac{p_{i}-2}{(p_{i}-1)(N_{i}-1)(N_{i}-2)}
	\sum_{\substack{\ell,\ell'=1\\ \ell\neq \ell'}}^{N_{i}}
	K_{ii}(x_{i}^{k}-x_{i}^{\ell})K_{ii}(x_{i}^{k}-x_{i}^{\ell'}), \\
	J_5 &= \frac{1}{(p_{i}-1)(N_{i}-1)}\sum_{\ell=1}^{N_{i}}
	K_{ii}(x_{i}^{k}-x_{i}^{\ell})^2.
\end{align*}
For the term $(\E(f_i^k))^2=(\E(g_i^k))^2$, we expand the square:
\begin{align*}
  (\E(f_i^k))^2 &= \bigg(\sum_{j=1}^n\alpha_{ij}\sum_{\substack{\ell=1 \\
	(i,k)\neq(j,\ell)}}K_{ij}(x_i^k-x_j^\ell)\sum_{j'=1}^n\alpha_{ij'}
	\sum_{\substack{\ell'=1 \\	(i,k)\neq(j',\ell')}}K_{ij'}(x_i^k-x_{j'}^{\ell'})
	\bigg)^2 \\
	&= \widehat{J}_1 + \cdots + \widehat{J}_5, \quad\mbox{where} \\
  \widehat{J}_1 
	&= \sum_{\substack{j,j'=1\\ j,j'\neq i,\,j\neq j'}}^{n}\frac{1}{N_{j}N_{j'}}
	\sum_{\ell,\ell'=1}^{N_{j},N_{j'}}K_{ij}(x_{i}^{k}-x_{j}^{\ell})
	K_{ij'}(x_{i}^{k}-x_{j'}^{\ell'}), \\
	\widehat{J}_2 &= \sum_{j=1,\, j\neq i}^{n}\frac{1}{N_{j}^2}\sum_{\ell,\ell'=1}^{N_{j}}
	K_{ij}(x_{i}^{k}-x_{j}^{\ell})K_{ij}(x_{i}^{k}-x_{j}^{\ell'}), \\
	\widehat{J}_3 &= 2\sum_{j=1,\, j\neq i}^{n}\frac{1}{(N_{i}-1)N_{j}}
	\sum_{\ell=1}^{N_j}\sum_{\substack{\ell'=1\\ \ell'\neq k}}^{N_{i}}
	K_{ij}(x_{i}^{k}-x_{j}^{\ell})K_{ii}(x_{i}^{k}-x_{i}^{\ell'}), \\
	\widehat{J}_4 
	&= \frac{1}{(N_{i}-1)^2)}\sum_{\substack{\ell,\ell'=1\\ \ell\neq \ell'}}^{N_{i}}
	K_{i,i}(x_{i}^{k}-x_{i}^{\ell})K_{ii}(x_{i}^{k}-x_{i}^{\ell'}), \\
	\widehat{J}_5 
	&= \frac{1}{(N_{i}-1)^2}\sum_{\ell=1}^{N_{i}}K_{ii}(x_{i}^{k}-x_{i}^{\ell})^2.
\end{align*}
The variance of $f_i^k$ is the difference $(J_1+\cdots+J_5)
-(\widehat{J}_1+\cdots+\widehat{J}_5)$.
We observe that $J_3-\widehat{J}_3=0$ and that 
$$
  \frac{\min\{b_{i},b_{j},b_{j'}\} b_{i}}{\min\{b_{i},b_{j}\}
	\min\{b_{i},b_{j'}\}} = \frac{\max\{b_{i},b_{j},b_{j'}\}}{\max\{b_{j},b_{j'}\}},
$$
A tedious but straightforward computation yields for the other terms:
\begin{align*}
  &\operatorname{Var}(f_i^k) = \E(f_i^k)^2 - (\E f_i^k)^2 
	= (J_1-\widehat{J}_1) + (J_2+J_2'-\widehat{J}_2) 
	+ (J_4+J_5-\widehat{J}_4-\widehat{J}_5) \\
	&= \sum_{\substack{j,j'=1\\ j,j'\neq i,\,j\neq j'}}^{n}
	\bigg(\frac{\max\{b_{i},b_{j},b_{j'}\}}{\max\{b_{j},b_{j'}\}}-1\bigg)
	\frac{1}{N_{j}N_{j'}}\sum_{\ell=1}^{N_j}\sum_{\ell'=1}^{N_{j'}}
	K_{ij}(x_{i}^{k}-x_{j}^{\ell})K_{ij'}(x_{i}^{k}-x_{j'}^{\ell'}) \\
  &\phantom{xx}{}+ \sum_{j=1,\, j\neq i}^{n}\bigg(
	\frac{b_{i}-\min\{b_{i},b_{j}\}}{\min\{b_{i},b_{j}\}}
	- \frac{1}{N_{j}} + \frac{b_{i}}{\min\{b_{i},b_{j}\}p_{j}}\bigg)
	\frac{1}{N_{j}(N_{j}-1)} \\
	&\phantom{xxxx}{}\times\sum_{\ell,\ell'=1,\,\ell\neq \ell'}^{N_{j}}
	K_{ij}(x_{i}^{k}-x_{j}^{\ell})K_{ij}(x_{i}^{k}-x_{j}^{\ell'}) \\
	  &\phantom{xx}{}+ \sum_{j=1,\, j\neq i}^{n}\bigg(
	\frac{\max\{b_{i},b_{j}\}-1}{N_{j}}\bigg)
	\frac{1}{N_{j}}\sum_{\ell=1}^{N_{j}}K_{ij}(x_{i}^{k}-x_{j}^{\ell})^2 \\
  &\phantom{xx}{}+ \bigg(\frac{1}{p_{i}-1}-\frac{1}{N_{i}-1}\bigg)\frac{1}{N_{i}-2}
	\sum_{\substack{\ell=1\\k\neq \ell}}^{N_{i}}\bigg(K_{ii}(x_{i}^k-x_{j}^{\ell})
	- \frac{1}{N_{i}-1}\sum_{\substack{\ell'=1\\ \ell'\neq k}}^{N_{i}}
	K_{ii}(x_{i}^k-x_{j}^{\ell'})\bigg)^2.
\end{align*}
The right-hand side equals \eqref{2.var}, which finishes the proof.
\end{proof}

For later use, we prove the following auxiliary result, which generalizes 
Lemma 3.2 in \cite{JLL20} to the multi-species case.

\begin{lemma}\label{lem.S}
Let $i\in\{1,\ldots,n\}$, $k\in\{1,\ldots,N_i\}$, and $(i,k)\in\C_{i,r}$ for
some $r=a(i,k)\le b_i$. Let $S_j^\ell\in\R^d$ with $j$, $\ell\in\N$ be random
variables which are independent of the partitioning random variable $\xi_m$.
Then it holds
\begin{align*}
  \bigg\|\frac{1}{p_j}\sum_{\ell\in\C_{j,r}}S_j^\ell\bigg\|
	&= \max_{\ell=1,\ldots,N_j}\|S_j^\ell\|\quad\mbox{if }i\neq j, \\
	\bigg\|\frac{1}{p_j-1}\sum_{\ell\in\C_{i,r},\,\ell\neq k}S_i^\ell\bigg\|
	&= \max_{\ell=1,\ldots,N_i}\|S_i^\ell\|\quad\mbox{if }i = j, 
\end{align*}
recalling that $\|\cdot\|=(\E(\cdot)^2)^{1/2}$.
\end{lemma}

\begin{proof}
The proof is similar to that one of \cite[Lemma 3.2]{JLL20}. We present it for
completeness. Let $i\neq j$ and 
set $I_i^k(j,\ell)=1$ if $(i,k)$ and $(j,\ell)$ are in same batch
and $I_i^k(j,\ell)=0$ otherwise. Due to the independency of $S_j^\ell$ and $\xi_m$,
we have
\begin{align*}
  \bigg\|\frac{1}{p_j}\sum_{\ell\in\C_{j,r}}S_j^\ell\bigg\|^2
	&= \frac{1}{p_j^2}\E\bigg(\sum_{\ell=1}^{N_j}I_i^k(j,\ell)S_j^\ell\bigg)^2
	= \frac{1}{p_j^2}\sum_{\ell,\ell'=1}^{N_j}\E\big(I_i^k(j,\ell)I_i^k(j,\ell')
	S_j^\ell S_j^{\ell'}\big) \\
	&= \frac{1}{p_j^2}\sum_{\ell,\ell'=1}^{N_j}\E(I_i^k(j,\ell)I_i^k(j,\ell'))
	\E(S_j^\ell S_j^{\ell'}).
\end{align*}
We know from \eqref{2.EI} that $\E(I_i^k(j,\ell)I_i^k(j,\ell'))\le 1/b_j$ 
in the case of $\ell = \ell'$ and from Case 2 of Proposition \ref{prop.cons} that 
$\E(I_i^k(j,\ell)I_i^k(j,\ell'))\le (p_{j}-1)/(b_{j}(N_{j}-1))$, if 
$\ell \ne \ell'$. Therefore, using the Cauchy--Schwarz inequality and the fact 
that $N_j=b_jp_j$,
\begin{align*}
  \bigg\|\frac{1}{p_j}\sum_{\ell\in\C_{j,r}}S_j^\ell\bigg\|^2
	&\le \frac{1}{p_j^2}\bigg(\sum_{\ell,\ell'=1,\,\ell\neq \ell'}^{N_j}
	\frac{p_{j}-1}{b_j(N_{j}-1)}\|S_j^\ell\|\,\|S_j^{\ell'}\|
	+ \sum_{\ell=1}^{N_{j}}\frac{1}{b_{j}}\|S_j^\ell\|^2\bigg) \\
	&\le \max_{\ell=1,\ldots,N_j}\|S_j^\ell\|^2
	\bigg(\frac{(p_{j}-1)(N_{j}-1)N_{j}}{p_{j}N_{j}(N_{j}-1)}+\frac{1}{p_{j}}\bigg)
	\le \max_{\ell=1,\ldots,N_j}\|S_j^\ell\|^2.
\end{align*}
The case $i=j$ is shown in a similar way.
\end{proof}

The next result is concerned with the stability of $X_i^k$ and $\widetilde{X}_i^k$.

\begin{lemma}[Stability]\label{lem.stab1}
Let $q\ge 2$, and $X_{0,i}^k\in L^q(\Omega)$,
where $i\in\{1,\ldots,n\}$ and $k\in\{1,\ldots,N_i\}$. Then there exist constants
$C(q)$, $C_1>0$, independent of $(p_i,b_i)_{i=1,\ldots,n}$, $m$, and $T$, such that
\begin{align}
  & \sup_{t>0}\E|X_i^k(t)|^q \le C(q), \quad
	\sup_{t>0}\E|\widetilde{X}_i^k(t)|^q \le C(q)(1+\theta^q), \label{3.estX} \\
	& \sup_{t_{m-1}<t<t_m}\E\big(|\widetilde{X}_i^k(t)|^q\big|\F_{m-1}\big)
	\le |\widetilde{X}_i^k(t_{m-1})|^q + C(q)(1+\theta^q), \label{3.estXt}
\end{align}
where $\theta$ is defined in \eqref{1.theta}. Furthermore, it holds that
\begin{equation}\label{3.estXXt}
  \big|\E\big(\widetilde{X}_i^k(t)-\widetilde{X}_i^k(t_{m-1})\big|\F_{m-1}\big)\big|
	\le C_V\tau|\widetilde{X}_i^k(t_{m-1})|^{\widetilde{q}_i} +
	C_1\tau(1+\theta^{\widetilde{q}_i}),
\end{equation}
where $\widetilde{q}_i=\max\{2,q_i\}$ and $C_V>0$ is introduced in Assumption (A2).
\end{lemma}

\begin{proof}
Let $i\in\{1,\ldots,n\}$ and $k\in\{1,\ldots,N_i\}$ be arbitrary but fixed. 
The proof is similar to \cite[Lemma 3.3]{JLL20} with the exception that we
work out the dependence on the number of batches $b_i$ in terms of the quotient
$\theta$.

{\em Step 1: Stability for $X_i^k(t)$.} Let $d\ge 2$.
We use It\^o's calculus for the process
$|X_i^k|^q$ and apply the expectation as in \cite[Lemma 3.3]{JLL20}, which yields
\begin{align*}
  \frac{\d}{\d t}\E|X_i^k(t)|^q &= -q\E\big(|X_i^k(t)|^{q-2}X_i^k(t)
	\cdot\na V_i(X_i^k(t))\big) \\
	&\phantom{xx}{}+ \frac{q}{N_i-1}\E\bigg(|X_i^k(t)|^{q-2}X_i^k(t)\cdot
	\sum_{\ell=1,\,\ell\neq k}^{N_i}K_{ii}(X_i^k(s)-X_i^\ell(s))\bigg) \\
	&\phantom{xx}{}
	+ \sum_{j=1,\,j\neq i}^n\frac{q}{N_j}\E\bigg(|X_i^k(t)|^{q-2}X_i^k(t)\cdot
	\sum_{\ell=1}^{N_j}K_{ij}(X_i^k(s)-X_j^\ell(s))\bigg) \\
	&\phantom{xx}{}+ \frac{\sigma_i^2}{2}q(q+d-2)\E|X_i^k(t)|^{q-2}.
\end{align*}
The mean-value theorem with intermediate value $\zeta\in\R^d$ and the 
convexity of $x\mapsto V_i(x)-r_i|x|^2/2$ (Assumption (A2)) imply that for all
$x\in\R^d$,
$$
  x\cdot\na V_i(x) = x^T D^2V_i(\zeta)x + x\cdot\na V_i(0) 
  \ge r_i|x|^2 + x\cdot\na V_i(0).
$$
Together with Fubini's theorem, the boundedness of the kernels $K_{ij}$
(Assumption (A1)), and Young's inequality, it follows that
\begin{align*}
  \frac{\d}{\d t}\E|X_i^k(t)|^q &\le -qr_i\E|X_i^k(t)|^q
	+ q\bigg(|\na V_i(0)|  \sum_{j=1}^n\|K_{ij}\|_\infty\bigg)\E|X_i^k(t)|^{q-1} \\
	&\phantom{xx}{}+ \frac{\sigma_i^2}{2}q(q+d-2)\E|X_i^k(t)|^{q-2}
	\le -\frac{qr_i}{2}\E|X_i^k(t)|^q + C_2,
\end{align*}
where $C_2>0$ depends on $\na V_i$, $K_{ij}$, $\sigma_i$, $d$, and $q$.
Gronwall's lemma implies that $\E|X_i^k(t)|^q$ is bounded by a constant
depending on $q$ (and not depending on $T$). 

{\em Step 2: Stability for $\widetilde{X}_i^k$.} Let $t\in(t_{m-1},t_m]$
and let $(i,k)\in\C_r$ for some $r\in\N$. Similarly as in the previous step,
we use It\^o's calculus and apply the conditional expectation with respect to
$\F_{m-1}$, observing that $|\widetilde{X}_i^k(t_{m-1})|^q$ is $\F_{m-1}$-measurable.
Then, applying Lemmas \ref{lem.cond1} and \ref{lem.cond2} in the appendix,
\begin{align*}
  \frac{\d}{\d t}&\E\big(|\widetilde{X}_i^k(t)|^q\big|\F_{m-1}\big)
	= -q\E\big(|\widetilde{X}_i^k(t)|^{q-2}\widetilde{X}_i^k
	\cdot\na V_i(\widetilde{X}_i^k(t))\big|\F_{m-1}\big) \\
	&{}+ \frac{q}{p_i-1}\sum_{\ell\in\C_{i,r},\,\ell\neq k}
	\E\big(|\widetilde{X}_i^k(t)|^{q-2}	\widetilde{X}_i^k(t)\cdot 
	K_{ii}(\widetilde{X}_i^k-\widetilde{X}_i^\ell)\big|\F_{m-1}\big) \\
	&{}+ \sum_{j=1,\,j\neq i}^n\frac{qb_i}{p_j\min\{b_i,b_j\}}\sum_{\ell\in\C_{j,r}}
	\E\big(|\widetilde{X}_i^k(t)|^{q-2}\widetilde{X}_i^k(t)\cdot 
	K_{ij}(\widetilde{X}_i^k(t)-\widetilde{X}_j^\ell(t))\big|\F_{m-1}\big) \\
	&{}+ \frac{\sigma_i^2}{2}q(q+d-2)
	\E\big(|\widetilde{X}_i^k(t)|^{q-2}\big|\F_{m-1}\big).
\end{align*}
Proceeding as in the previous step and using $b_i/\min\{b_i,b_j\}\le\theta$,
we infer that
$$
  \frac{\d}{\d t}\E\big(|\widetilde{X}_i^k(t)|^q\big|\F_{m-1}\big)
	\le -\frac{qr_i}{2}\E\big(|\widetilde{X}_i^k(t)|^{q}\big|\F_{m-1}\big) 
	+ C_3(1+\theta)^q,
$$
and Gronwall's lemma on $(t_{m-1},t_m]$ implies \eqref{3.estXt}. 
Finally, the second estimate in \eqref{3.estX} is proved in a similar way,
using the Gronwall lemma on $[0,t]$ and taking into account that
$\E|X_{0,i}^k|^q$ is bounded by assumption.

{\em Step 3: Proof of estimate \eqref{3.estXXt}.} We apply It\^o's lemma, 
take the conditional expectation
of $\widetilde{X}_i^k(t)-\widetilde{X}_i^k(t_{m-1})$, and use the polynomial growth
condition for $\na V_i$ in Assumption (A2) as well as the boundedness of $K_{ij}$:
\begin{align*}
  \E\big(&\widetilde{X}_i^k(t)-\widetilde{X}_i^k(t_{m-1})\big|\F_{m-1}\big)
	= -\int_{t_{m-1}}^t\E\big(\na V_i(\widetilde{X}_i^k(s))\big|\F_{m-1}\big)\d s \\
	&\phantom{xx}{}
	+ \frac{1}{p_i-1}\int_{t_{m-1}}^t\E\bigg(\sum_{\ell\in\C_{i,r},\,\ell\neq k}
	K_{ii}(\widetilde{X}_i^k(s)-\widetilde{X}_i^\ell(s))\big|\F_{m-1}\bigg)\d s \\
	&\phantom{xx}{}
	+ \sum_{j=1,\, j\neq i}\frac{b_i}{p_j\min\{b_i,b_j\}}\int_{t_{m-1}}^t
	\E\bigg(\sum_{\ell\in\C_{j,r}}K_{ij}(\widetilde{X}_i^k(s)-\widetilde{X}_j^\ell(s))
	\big|\F_{m-1}\bigg)\d s \\
  &\le C_V\tau + C_V\int_{t_{m-1}}^t\E\big(|\widetilde{X}_i^k(s)|^{q_i}\big|\F_{m-1}
	\big)\d s + \tau\sum_{j=1}^n\frac{\|K_{ij}\|_\infty b_i}{\min\{b_i,b_j\}}.
\end{align*}
It follows from \eqref{3.estXt} with $q=\widetilde{q}_i:=\max\{2,q_i\}$ that
\begin{align*}
  \E\big(&\widetilde{X}_i^k(t)-\widetilde{X}_i^k(t_{m-1})\big|\F_{m-1}\big) \\
	&\le C_V\tau|\widetilde{X}_i^k(t_{m-1})|^{\widetilde{q}_i}
	+ \tau\bigg(C_V + C_V C(\widetilde{q}_i)(1+\theta^{\widetilde{q}_i}) + \sum_{j=1}^n
	\frac{\|K_{ij}\|_\infty b_i}{\min\{b_i,b_j\}}\bigg) \\
	&\le C_V\tau|\widetilde{X}_i^k(t_{m-1})|^{\widetilde{q}_i}
	+ C_4(\widetilde{q}_i)(1+\theta^{\widetilde{q}_i}).
\end{align*}
This completes the proof.
\end{proof}


\section{Control of the error process}\label{sec.est}

We prove first a bound for the difference
$\widetilde{X}_i^k(t)-\widetilde{X}_i^k(t_{m-1})$.

\begin{lemma}\label{lem.stab2}
Let $t\in(t_{m-1},t_m]$, let $\widetilde{X}$ be the stochastic process defined in
\eqref{1.rbm}, and let $i\in\{1,\ldots,n\}$. Set $q_i'=2\max\{1,q_i\}$, where
$q_i$ is defined in Assumption (A2). Then, for any $(i,k)\in\C_r$ for some $r\le b_i$
such that $X_{0,i}^k\in L^{q_i'}(\Omega)$, there exists a constant $C>0$, independent of
$(p_i,b_i)_{i=1,\ldots,n}$, and $\xi_m$, such that
$$
  \E\big(|\widetilde{X}_i^k(t)-\widetilde{X}_i^k(t_{m-1})|^2\big|\F_{m-1}\big)
	\le C\tau(1+\theta^{q_i'/2+1})\big(1 + |\widetilde{X}_i^k(t_{m-1})|^{q'_i/2+1}\big).
$$
\end{lemma}

\begin{proof}
Again, the proof is similar to \cite[Lemma 3.3]{JLL20} and based on It\^o's calculus.
Let $t\in(t_{m-1},t_m]$ and $(i,k)\in\C_r$ for some $r\le b_i$, satisfying the
assumptions of the lemma. Set $S(t):=\widetilde{X}_i^k(t)-\widetilde{X}_i^k(t_{m-1})$.
We apply It\^o's lemma to $|S(t)|^2$ and the conditional expectation and use
Lemmas \ref{lem.cond1} and \ref{lem.cond2}:
\begin{align}
  \E&(|S(t)|^2|\F_{m-1}) \le 2\int_{t_{m-1}}^t\big|\E\big(S(s)\cdot\na V_i
	(\widetilde{X}_i^k(s))\big|\F_{m-1}\big)\big|\d s + d\int_{t_{m-1}}^t\sigma_i^2\d s 
	\label{3.aux1} \\
	&\phantom{xx}{}
	+ \frac{2}{p_i-1}\int_{t_{m-1}}^t\bigg|\E\bigg(\sum_{\ell\in\C_{i,r},\,\ell\neq k}
	K_{ii}(\widetilde{X}_i^k(s)-\widetilde{X}_i^\ell(t))\cdot S(s)\bigg|\F_{m-1}\bigg)
	\bigg|\d s \nonumber \\
	&\phantom{xx}{}
	+ \sum_{j=1,\,j\neq i}\frac{2b_i}{p_j\min\{b_i,b_j\}}\int_{t_{m-1}}^t
	\bigg|\E\bigg(\sum_{\ell\in\C_{i,r},\,\ell\neq k}
	K_{ij}(\widetilde{X}_i^k(s)-\widetilde{X}_j^\ell(t))\cdot S(s)\bigg|\F_{m-1}\bigg)
	\bigg|\d s \nonumber \\
  &=: J_6+\cdots+J_9. \nonumber
\end{align}
By the Cauchy--Schwarz inequality, the polynomial growth condition on $\na V_i$
(Assumption (A2)), and stability estimate \eqref{3.estX} with 
$q=q'_i$, we have
\begin{align*}
  J_6 &\le 2C_V^{1/2}\int_{t_{m-1}}^t\big(\E(|S(s)|^2\big|\F_{m-1})\big)^{1/2}
	\big(\E(1+|\widetilde{X}_i^k(s)|^{2q_i}\big|\F_{m-1})\big)^{1/2}\d s \\
	&\le 2C_V^{1/2}\big(1 + C(q)(1+\theta^q) + |\widetilde{X}_i^k(t_{m-1})|^q\big)^{1/2}
	\int_{t_{m-1}}^t\big(\E(|S(s)|^2\big|\F_{m-1})\big)^{1/2}\d s.
\end{align*}
Next, using the boundedness of $K_{ii}$, Lemma \ref{lem.cond0}, 
and H\"{o}lder's inequality,
\begin{align*}
  J_8 &\le C\|K_{ii}\|_\infty\int_{t_{m-1}}^t
	\big(\E(|S(s)|^2\big|\F_{m-1})\big)^{1/2}\d s, \\
  J_9 &\le C\sum_{j=1,\,j\neq i}^n\frac{b_i}{\min\{b_i,b_j\}}\|K_{ii}\|_\infty
	\int_{t_{m-1}}^t\big(\E(|S(s)|^2\big|\F_{m-1})\big)^{1/2}\d s \\
	&\le C\theta\int_{t_{m-1}}^t\big(\E(|S(s)|^2\big|\F_{m-1})\big)^{1/2}\d s.
\end{align*}
Hence, we infer from \eqref{3.aux1} that
\begin{equation}\label{3.aux2}
  \E(|S(t)|^2|\F_{m-1}) \le C_5\int_{t_{m-1}}^t
	\big(\E(|S(s)|^2|\F_{m-1})\big)^{1/2}\d s + d\sigma_i^2(t-t_{m-1}),
\end{equation}
where $C_5:=2C_V^{1/2}(1 + C(q)(1+\theta^q) + |\widetilde{X}_i^k(t_{m-1})|^q)^{1/2} 
+ C\theta$. We deduce from estimate \eqref{3.estX} that the integrand on the right-hand 
side can be estimated according to
\begin{align*}
  \E(|S(s)|^2|\F_{m-1}) 
	&\le \frac12\E\big(|\widetilde{X}_i^k(s)|^2\big|\F_{m-1}\big)
	+ \frac12|\widetilde{X}_i^k(t_{m-1})|^2 \\
	&\le \frac{C(2)}{2}(1+\theta^2) + \frac12|\widetilde{X}_i^k(t_{m-1})|^2.
\end{align*}
Inserting this estimate into \eqref{3.aux2}, we conclude that
$$
  \E(|S(t)|^2|\F_{m-1}) \le C_6\tau(1+\theta^{q/2+1})
	\big(1 + |\widetilde{X}_i^k(t_{m-1})|^{q/2+1}\big),
$$
where $C_6>0$ does not depend on $b_i$, $p_i$, or $\xi_m$.
\end{proof}

We define the error process $Z_i^k(t):=\widetilde{X}_i^k(t)-X_i^k(t)$ and
prove some estimates for $Z_i^k(t)$, generalizing \cite[Lemma 3.4]{JLL20}.

\begin{lemma}[Control of the error process]\label{lem.Z}
Let $i\in\{1,\ldots,n\}$, $k\in\{1,\ldots,N_i\}$, and $X_{0,i}^k\in L^{q'_i}(\Omega)$,
where $q'_i=2\max\{1,q_i\}$ and $q_i$ is introduced in Assumption (A2). Then there
exists a constant $C>0$, which is independent of $(b_i,p_i)_{i=1,\ldots,n}$, 
and $m$ such that for all $t\in(t_{m-1},t_m]$,
\begin{align}
  & \|Z_i^k(t)-Z_i^k(t_{m-1})\| \le C\tau(1+\theta^{q_i'/2}), \quad
	|Z_i^k(t)|\le C\tau\theta + |Z_i^k(t_{m-1})|, \label{3.Z1} \\
	& \big|\E\big((Z_i^k(t)-Z_i^k(t_{m-1}))\chi_i^k(\widetilde{X}(t_{m-1}))\big)\big| 
	\label{3.Z2} \\
	&\phantom{xx}{}\le C\tau\bigg((1+\theta^{3q'_i/2})\tau 
	+ (1+\theta^{q'_i})\|Z_i^k(t)\|
	+ \sum_{j=1}^n\|Z_j^1(t)\|\bigg) 
	+ 8\tau\max_{j=1,\ldots,n}\|K_{ij}\|_\infty^2\Gamma_i, \nonumber
\end{align}
where $\Gamma_i$ and $\chi_i^k$ are defined in \eqref{1.Gamma} and \eqref{2.chi},
respectively.
\end{lemma}

\begin{proof} Since the Brownian motions are the same for $X_i^k$ and 
$\widetilde{X}_i^k$, the process $Z_i^k(t)$ solves for $t\in(t_{m-1},t_m]$
the deterministic equation
\begin{align}\label{3.eqZ}
  \d Z_i^k(t) &= -\big(\na V_i(\widetilde{X}_i^k(t))-\na V_i(X_i^k(t))\big)\d t
	+ \sum_{j=1}^n\beta_{ij}\sum_{\ell\in\C_{j,r},\,(i,k)\neq(j,\ell)}
	K_{ij}(\widetilde{X}_i^k-\widetilde{X}_j^\ell)\d t \\
	&\phantom{xx}{}- \sum_{j=1}^n\alpha_{ij}\sum_{\ell=1,\,(i,k)\neq(j,\ell)}^{N_j}
	K_{ij}(X_i^k-X_j^\ell)\d t. \nonumber
\end{align} 

{\em Step 1: Proof of \eqref{3.Z1}.} Let $(i,k)\in\C_{i,r}$
for some $r\le b_i$. We take the expectation of the difference of the equations
\eqref{3.eqZ} solved by $Z_i^k(t)$ and $Z_i^k(t_{m-1})$, respectively, and
distinguish the cases $j=i$ and $j\neq i$, leading to
\begin{align}
  & \|Z_i^k(t)-Z_i^k(t_{m-1})\| \le J_{10}+\cdots+J_{14}, \quad\mbox{where} 
	\label{3.aux3} \\
	& J_{10} = \bigg\|\int_{t_{m-1}}^t
	\big(\na V_i(\widetilde{X}_i^k(s))-\na V_i(X_i^k(s))\big)\d s\bigg\|, \nonumber \\
	& J_{11} = \frac{1}{p_i-1}\bigg\|\int_{t_{m-1}}^t
	\sum_{\ell\in\C_{j,r},\,(i,k)\neq(j,\ell)}
  K_{ii}(\widetilde{X}_i^k(s)-\widetilde{X}_i^\ell(s))\d s\bigg\|,  \nonumber \\
	& J_{12} = \sum_{j=1,\,j\neq i}^n\frac{b_i}{p_j\min\{b_i,b_j\}}\bigg\|\int_{t_{m-1}}^t
	\sum_{\ell\in\C_{j,r}}K_{ij}(\widetilde{X}_i^k(s)-\widetilde{X}_j^\ell(s))
	\d s\bigg\|,  \nonumber \\
	& J_{13} = \frac{1}{N_i-1}\sum_{\ell=1,\,\ell\neq k}^{N_j}\bigg\|\int_{t_{m-1}}^t
	K_{ii}(X_i^k(s)-X_i^\ell(s))\d s\bigg\|,  \nonumber \\
	& J_{14} = \sum_{j=1,\,j\neq i}^n\frac{1}{N_j}\sum_{\ell=1}^{N_j}\bigg\|
	\int_{t_{m-1}}^t K_{ij}(X_i^k(s)-X_j^\ell(s))\d s\bigg\|.  \nonumber 
\end{align}
For the first term, we use the Cauchy--Schwarz inequality, the growth condition
of $\na V_i$, and stability estimate \eqref{3.estX} with $q=q'_i$: 
\begin{align*}
  J_{10} &\le \sqrt{\tau}\bigg(\E\int_{t_{m-1}}^t
	\big|\na V_i(\widetilde{X}_i^k(s))-\na V_i(X_i^k(s))\big|^2\d s\bigg)^{1/2} \\
	&\le C_V\sqrt{\tau}\bigg(\E\int_{t_{m-1}}^t \big(1 + \E|\widetilde{X}_i^k(s)|^{q}
	+ \E|X_i^k(s)|^{q}\big)\d s\bigg)^{1/2} 
	\le C\tau(1+\theta^{q/2}).
\end{align*}
For the remaining terms, we exploit the boundedness of $K_{ij}$, yielding
$$
  J_{11}+\cdots+J_{14} \le C\tau\sum_{j=1}^n\bigg(1+\frac{b_i}{\min\{b_i,b_j\}}\bigg)
	\|K_{ij}\|_\infty \le C\tau(1+\theta).
$$
Thus, we deduce from \eqref{3.aux3} that
$$
  \|Z_i^k(t)-Z_i^k(t_{m-1})\| \le C\tau(1+\theta+\theta^{q/2}),
$$
which proves the first inequality in \eqref{3.Z1}.

We estimate similarly as in the proof
of Lemma \ref{lem.stab1}, using the strong convexity of $V_i$ and the boundedness
of $K_{ij}$:
$$
  \frac{\d}{\d t}|Z_i^k(t)|^2 \le -r_i|Z_i^k(t)|^2 
	+ \sum_{j=1}^n\frac{b_i\|K_{ij}\|_\infty}{\min\{b_i,b_j\}}|Z_i^k(t)| 
	\le C\theta|Z_i^k(t)|.
$$
This implies after integration with respect to time that $|Z_i^k(t)|\le 
C\theta\tau+|Z_i^k(t_{m-1})|$, showing the second inequality in \eqref{3.Z1}.

{\em Step 2: Proof of \eqref{3.Z2}.} Set
$\Delta K_{ij}^\ell:=K_{ij}(\widetilde{X}_i^k-\widetilde{X}_j^\ell)
- K_{ij}(X_i^k-X_j^\ell)$. Using the formulation \eqref{1.rbm2} for 
$\widetilde{X}_i^k$, we find that
\begin{align}\label{3.J15J18}
  & \big|\E\big((Z_i^k(t)-Z_i^k(t_{m-1}))\chi_i^k(X(t))\big)\big|
	\le J_{15}+\cdots+J_{18}, \quad\mbox{where} \\
	& J_{15} = \E\bigg(\int_{t_{m-1}}^{t}\big|\na V_{i}(\widetilde{X}_i^k(s)) 
	- \na V_{i}(X_i^k(s))\big|\d s |\chi_{i}^{k}(X(t))|\bigg), \nonumber \\
	& J_{16} = \E\bigg(\frac{1}{p_i-1}\int_{t_{m-1}}^t\sum_{\ell\in\C_{i,r},\,\ell\neq k}
	|\Delta K_{ii}^\ell(s)|\d s|\chi_i^k(X(t))|\bigg), \nonumber \\
	& J_{17} = \E\bigg(\sum_{j=1,\,j\neq i}^n\frac{b_i}{p_j\min\{b_i,b_j\}}
	\int_{t_{m-1}}^t\sum_{\ell\in\C_{j,r}}|\Delta K_{ij}^\ell(s)|\d s
	|\chi_i^k(X(t))|\bigg), \nonumber \\
	& J_{18} = \E\bigg|\int_{t_{m-1}}^t\chi_i^k(X(s))\d s\cdot\chi_i^k(X(t))\bigg|.
	\nonumber 
\end{align}
For the term $J_{15}$, we use the mean-value theorem and the growth condition
for $D^2V_i$ (Assumption (A2)):
\begin{align*}
  \big|\na &V_i(\widetilde{X}_i^k-\na V_i(X_i^k)\big|
	\le |\widetilde{X}_i^k-X_i^k|\int_0^1|D^2 V_i
	(\widetilde{X}_i^k-\eta(\widetilde{X}_i^k-X_i^k))|\d\eta \\
	&\le C_V|Z_i^k|\int_0^1\big(1 + |\widetilde{X}_i^k - \eta(\widetilde{X}_i^k-X_i^k)
	|^{q_i}\big)d\eta 
	\le C|Z_i^k|(1+ |\widetilde{X}_i^k|^{q_i} + |X_i^k|^{q_i}).
\end{align*}
Since $|\chi_i^k|\le 2\sum_{j=1}^n\|K_{ij}\|_\infty$, the Cauchy--Schwarz inequality
and stability estimate \eqref{3.estX} lead to
\begin{align*}
  J_{15} &\le C\sum_{j=1}^n\|K_{ij}\|_\infty\int_{t_{m-1}}^t\E\big(|Z_i^k(s)|
	(1+ |\widetilde{X}_i^k(s)|^{q_i} + |X_i^k(s)|^{q_i})\big)\d s \\
	&\le C\sum_{j=1}^n\|K_{ij}\|_\infty\int_{t_{m-1}}^t\|Z_i^k(s)\|
	\|1+ |\widetilde{X}_i^k(s)|^{q_i} + |X_i^k(s)|^{q_i}\|\d s \\
	&\le C(q)\tau(1+\theta^{q_i})\big(\tau(1+\theta^{q_i/2}) + \|Z_i^k(t)\|\big).
\end{align*}
The last inequality follows from
\begin{align}\label{3.triangle}
  \|Z_i^k(s)\| &\le \|Z_i^k(s)-Z_i^k(t_{m-1})\| 
	+ \|Z_i^k(t_{m-1})-Z_i^k(t)\| + \|Z_i^k(t)\| \\
	&\le 2C\tau(1+\theta^{q_i/2}) + \|Z_i^k(t)\|, \nonumber 
\end{align}
which in turn is a consequence of estimate \eqref{3.Z1}. We conclude that
$$
  J_{15} \le C\tau^2(1+\theta^{3q_i/2}) + C\tau(1+\theta^{q_i})\|Z_i^k(t)\|.
$$

We use the Lipschitz continuity of $K_{ij}$ (Assumption (A1)) to obtain
\begin{align*}
  J_{16} &\le \frac{2}{p_i-1}\sum_{j=1}^n\|K_{ij}\|_\infty\int_{t_{m-1}}^t\E
	\sum_{\ell\in\C_{i,r},\,\ell\neq k}|\Delta K_{ij}^\ell(s)|\d s \\
	&\le \frac{CL_{ii}}{p_i-1}\int_{t_{m-1}}^t\E\sum_{\ell\in\C_{i,r},\,\ell\neq k}
	\big(|\widetilde{X}_i^k(s)-X_i^k(s)| + |X_i^\ell(s)-\widetilde{X}_i^\ell(s)|\big)
	\d s \\
	&\le \frac{CL_{ii}}{p_i-1}\int_{t_{m-1}}^t\bigg(
	\bigg\|\sum_{\ell\in\C_{i,r},\,\ell\neq k}Z_i^k(s)\bigg\|
	+ \bigg\|\sum_{\ell\in\C_{i,r},\,\ell\neq k}Z_i^\ell(s)\bigg\|\bigg)\d s.
\end{align*}
It follows from the second estimate in \eqref{3.Z1}, i.e.\ $|Z_i^k(t)|\le C\tau\theta
+ |Z_i^k(t_{m-1})|$, that
$$
  J_{16} \le \frac{CL_{ii}}{p_i-1}\int_{t_{m-1}}^t\bigg((p_i-1)C\tau\theta
	+ (p_i-1)\|Z_i^k(s)\| + \bigg\|\sum_{\ell\in\C_{i,r},\,\ell\neq k}
	Z_i^\ell(t_{m-1})\bigg\|\bigg)\d s.
$$
The variable $Z_i^\ell(s)$ is $\G_{m-1}$-measurable for all $t_{m-1}<s<t$ and
hence it is independent of $\xi_{m-1}$. Therefore, we can apply Lemma \ref{lem.S}
to the last term of the integrand to find that
\begin{equation}\label{3.aux4}
  J_{16} \le CL_{ii}\int_{t_{m-1}}^t\big(C\tau\theta
	+ \|Z_i^k(s)\| + \|Z_i^k(t_{m-1})\|\big)\d s.
\end{equation}
Here, we have taken into account the fact that $\|Z_i^k(t)\|=\|Z_i^\ell(t)\|$
for every $k$, $\ell=1,\ldots,N_i$. The last two terms of the integrand can be
estimated, by estimate \eqref{3.Z1}, according to \eqref{3.triangle} and 
$$
	\|Z_i^k(t_{m-1})\| \le \|Z_i^k(t_{m-1})-Z_i^k(t)\| + \|Z_i^k(t)\|
	\le C\tau(1+\theta^{q_i/2}) + \|Z_i^k(t)\|.
$$
Hence, we conclude from \eqref{3.aux4} that
$$
  J_{16} \le C\tau\big(\tau(1+\theta^{q/2}) + \|Z_i^k(t)\|\big),
$$
where $C>0$ does not depend on $b_i$, $p_i$, or $m$ and recalling that
we have chosen $q=2\max\{1,q_i\}$. Similar arguments lead to
\begin{align*}
  J_{17} &\le 2\tau\sum_{j=1,\,j\neq i}^n\|K_{ij}\|_\infty
	\big(\tau(1+\theta^{q/2}) + \|Z_i^k(t)\| + \|Z_j^1(t)\|\big) \\
	&\le C\tau\bigg(\tau(1+\theta^{q/2}) + \|Z_i^k(t)\|
	+ \sum_{j=1,\,j\neq i}^n\|Z_j^1(t)\|\bigg).
\end{align*}

Finally, we estimate the remaining term. By the Cauchy--Schwarz inequality,
\begin{equation}\label{3.aux5}
  J_{18} = \int_{t_{m-1}}^t\E\big|\chi_i^k(X(s))\cdot\chi_i^k(X(t))\big|\d s
	\le \int_{t_{m-1}}^t\|\chi_i^k(X(s))\|\,\|\chi_i^k(X(t))\|\d s.
\end{equation}
By Lemma \ref{lem.cond0} in the appendix,
$$
  \|\chi_i^k(X(s))\|^2 = \E|\chi_i^k(X(s))|^2
	= \E\big[\E\big(|\chi_i^k(X(s))|^2\big|\sigma(X(s))\big) \big],
$$
where $\sigma(X(s))$ is the $\sigma$-algebra generated by $X(s)$.
Proposition \ref{prop.cons} states that $\E\chi_i^k(X(s))=0$ and
$\operatorname{Var}(\chi_i^k(X(s)))\le 8\max_{j=1,\ldots,n}\|K_{ij}\|_\infty^2\Gamma_i$.
Therefore,
$$
  \|\chi_i^k(X(s))\|^2 = \operatorname{Var}_{\sigma(X(s))}(\chi_i^k(X(s)))
	\le 8\max_{j=1,\ldots,n}\|K_{ij}\|_\infty^2\Gamma_i.
$$
Inserting this estimate into \eqref{3.aux5} leads to
$$
  J_{18} \le 8\tau\max_{j=1,\ldots,n}\|K_{ij}\|_\infty^2\Gamma_i.
$$
Summarizing, we obtain from \eqref{3.J15J18}
\begin{align*}
  \big|\E&\big((Z_i^k(t)-Z_i^k(t_{m-1}))\chi_i^k(X(t))\big)\big| \\
	&\le C(q)\tau\bigg(\tau(1+\theta^{3q/2}) + (1+\theta^{q})\|Z_i^k(t)\|
	+ \sum_{j=1}^n\|Z_j^1(t)\|\bigg)
	+ 8\max_{j=1,\ldots,n}\|K_{ij}\|_\infty^2\Gamma_i,
\end{align*}
which finishes the proof.
\end{proof}


\section{Proof of Theorem \ref{thm.main}}\label{sec.proof}

Let $i\in\{1,\ldots,n\}$ and $k\in\{1,\ldots,N_i\}$ be such that $(i,k)\in\C_{i,r}$
for some $r\le b_i$. As in the last section, we set
$\Delta K_{ij}^\ell:=K_{ij}(\widetilde{X}_i^k-\widetilde{X}_j^\ell)
- K_{ij}(X_i^k-X_j^\ell)$. The process $Z_i^k$ satisfies
\begin{align*}
  \d Z_i^k(t) &= -(\na V_i(\widetilde{X}_i^k(t))-\na V_i(X_i^k(t)))\d t
	+ \frac{1}{N_i-1}\sum_{\ell=1,\,\ell\neq k}^{N_i}\Delta K_{ii}^\ell(t)\d t \\
	&\phantom{xx}{}
	+ \sum_{j=1,\,j\neq i}^n \frac{1}{N_j}\sum_{\ell=1}^{N_j}\Delta K_{ij}^\ell(t)\d t
	+ \chi_i^k(\widetilde{X}(t))\d t.
\end{align*}
In particular, $Z_i^k$ is pathwise a.e.\ differentiable in time. 

{\em Step 1: Differential inequality for $|Z_i^k|^2$.} Together
with the strong convexity of $V_i$ (Assumption (A2)) and the Lipschitz continuity
of $K_{ij}$ (Assumption (A1)), we find that
\begin{align*}
  \frac12\frac{\d}{\d t}|Z_i^k|^2 &= -(\na V_i(\widetilde{X}_i^k)-\na V_i(X_i^k))
	\cdot Z_i^k + \frac{1}{N_i-1}\sum_{\ell=1,\,\ell\neq k}^{N_i}
	\Delta K_{ii}^\ell\cdot Z_i^k \\
	&\phantom{xx}{}+ \sum_{j=1,\,j\neq i}^n\frac{1}{N_j}\sum_{\ell=1}^{N_j}
	\Delta K_{ij}^\ell\cdot Z_i^k + \chi_i^k(\widetilde{X})\cdot Z_i^k \\
	&\le -r_i|Z_i^k|^2 + \frac{L_{ii}}{N_i-1}\sum_{\ell=1,\,\ell\neq k}^{N_i}
	(|Z_i^k|+|Z_i^\ell|)|Z_i^k| \\
	&\phantom{xx}{}+ \sum_{j=1,\,j\neq i}^n\frac{L_{ij}}{N_j}\sum_{\ell=1}^{N_j}
	(|Z_i^k|+|Z_j^\ell|)|Z_i^k| + \chi_i^k(\widetilde{X})\cdot Z_i^k.
\end{align*}
By taking the expectation and using Young's inequality, it follows after a standard
computation that
$$
  \frac12\frac{\d}{\d t}\E|Z_i^k|^2 \le -r_i\E|Z_i^k|^2
	+ \frac32\sum_{j=1}^n L_{ij}\E|Z_i^k|^2 + \frac12\sum_{j=1}^n L_{ij}|Z_j^k|^2
	+ \E(\chi_i^k(\widetilde{X})\cdot Z_i^k).
$$
Without loss of generality, we may take $k=1$ (since the distributions coincide).
A summation over $i=1,\ldots,n$ and exchanging the summation indices in the
third term of the right-hand side leads to
\begin{align}\label{3.dZi1dt}
  \frac12\frac{\d}{\d t}\sum_{i=1}^n\E|Z_i^1|^2 &\le -\sum_{i=1}^n r_i\E|Z_i^1|^2
	+ \frac32\sum_{i,j=1}^n L_{ij}\E|Z_i^1|^2 
	+ \frac12\sum_{i,j=1}^n L_{ji}\E|Z_i^1|^2 \\
	&\phantom{xx}{}+ \sum_{i=1}^n\E(\chi_i^1(\widetilde{X})\cdot Z_i^1) \nonumber \\
	&\le -\min_{i=1,\ldots,n}\bigg(r_i - 2\sum_{j=1}^n\max\{L_{ij},L_{ji}\}
	\bigg)\sum_{i=1}^n\|Z_i^1\|^2 + \E(\chi_i^1(\widetilde{X})\cdot Z_i^1). \nonumber
\end{align}

It remains to estimate the last term $\E(\chi(\widetilde{X})\cdot Z_i^1)$.
To this end, we write
\begin{align}
  & \E(\chi(\widetilde{X}(t))\cdot Z_i^1(t)) 
	= J_{19}+\cdots+J_{22}, \quad\mbox{where} \label{3.J19J22} \\
	& J_{19} = \E\big(Z_i^1(t_{m-1})\cdot\chi_i^1(\widetilde{X}(t_{m-1}))\big), 
	\nonumber \\
	& J_{20} = \E\big((Z_i^1(t)-Z_i^1(t_{m-1}))\cdot\chi_i^1(X(t))\big), \nonumber \\
	& J_{21} = \E\big(Z_i^1(t_{m-1})\cdot
	(\chi_i^1(\widetilde{X}(t))-\chi_i^1(\widetilde{X}(t_{m-1})))\big), \nonumber \\
	& J_{22} = \E\big((Z_i^1(t)-Z_i^1(t_{m-1}))\cdot(\chi_i^1(\widetilde{X}(t))
	- \chi_i^1(X(t)))\big). \nonumber 
\end{align}

{\em Step 2: Estimate of $J_{19}$ and $J_{20}$.} Since $\xi_{m,i}$ is independent 
of $\G_{m-1}$ and $Z_i^1(t_{m-1})$ is $\G_{m-1}$-measurable, we obtain from 
Lemma \ref{lem.cond0} in the appendix that
\begin{align*}
  \E\big(Z_i^1(t_{m-1})\big|\G_{m-1}\big) &= Z_i^1(t_{m-1}), \\
	\E\big(Z_i^1(t_{m-1})\cdot\chi(\widetilde{X}(t_{m-1}))\big|\G_{m-1}\big)
	&= Z_i^1(t_{m-1})\cdot\E\big(\chi(\widetilde{X}(t_{m-1}))\big|\G_{m-1}\big).
\end{align*}
This shows that, using Proposition \ref{prop.cons},
$$
  J_{19} = \E\big[\E\big(Z_i^1(t_{m-1})\cdot\chi(\widetilde{X}(t_{m-1}))
	\big|\G_{m-1}\big)\big]
	= \E\big[Z_i^1(t_{m-1})\cdot\E\big(\chi(\widetilde{X}(t_{m-1}))
	\big|\G_{m-1}\big)\big] = 0.
$$
The term $J_{20}$ can be directly estimated from \eqref{3.Z2}:
$$
  J_{20} \le C\tau\bigg((1+\theta^{3q'_i/2})\tau + (1+\theta^{q'_i})\|Z_i^1(t)\|
	+ \sum_{j=1}^n\|Z_j^1(t)\|\bigg) 
	+ 8\tau\max_{j=1,\ldots,n}\|K_{ij}\|_\infty^2\Gamma_i.
$$

{\em Step 3: Estimate of $J_{21}$.} We observe that $Z_i^1(t_{m-1})$ is
$\F_{m-1}$-measurable. By the law of total expectation (Lemma \ref{lem.cond0})
and the Cauchy--Schwarz inequality,
\begin{align}
  J_{21} &= \E\big[Z_i^1(t_{m-1})\E\big(\chi_i^1(\widetilde{X}_i^1(t))
	- \chi_i^1(\widetilde{X}(t_{m-1}))\big|\F_{m-1}\big)\big] \label{3.J21} \\
	&\le \|Z_i^1(t_{m-1})\|\,\big\|\E\big(\chi_i^1(\widetilde{X}_i^1(t))
	- \chi_i^1(\widetilde{X}(t_{m-1}))\big|\F_{m-1}\big)\big\|. \nonumber
\end{align}
We deduce from \eqref{3.triangle} that the first factor on the right-hand side 
is bounded from above by
\begin{equation}\label{3.Zi1} 
  \|Z_i^1(t_{m-1})\| \le C\tau(1+\theta^{q_i'/2}) + \|Z_i^k(t)\|.
\end{equation}
For the second factor, we introduce the notation
\begin{align*}
  \Delta \widetilde{K}_{ij}^\ell &:= 
	K_{ij}(\widetilde{X}_i^1(t)-\widetilde{X}_j^\ell(t))
	- K_{ij}(\widetilde{X}_i^1(t_{m-1})-\widetilde{X}_j^\ell(t_{m-1})), \\
	\Delta\widetilde{X}_{ij}^\ell &:= (\widetilde{X}_i^1(t)-\widetilde{X}_j^\ell(t))
	- (\widetilde{X}_i^1(t_{m-1})-\widetilde{X}_j^\ell(t_{m-1})).
\end{align*}
Since $\xi_m$ is $\F_{m-1}$-measurable, we can write the second factor on
the right-hand side of \eqref{3.J21} as follows:
\begin{align}
  \E&\big(\chi_i^1(\widetilde{X}(t))
	- \chi_i^1(\widetilde{X}(t_{m-1}))\big|\F_{m-1}\big) \label{3.aux6} \\
	&= \frac{1}{p_i-1}\sum_{\ell\in\C_{i,r},\,\ell\neq 1}\E(\Delta\widetilde{K}_{ii}^\ell
	|\F_{m-1}) - \frac{1}{N_i-1}\sum_{\ell=1,\,\ell\neq i}^{N_i}
	\E(\Delta\widetilde{K}_{ii}^\ell|\F_{m-1}) \nonumber \\
	&\phantom{xx}{}+ \sum_{j=1,\,j\neq i}^n\frac{b_i}{p_j\min\{b_i,bj\}}
	\sum_{\ell\in\C_{j,r}}\E(\Delta\widetilde{K}_{ij}^\ell|\F_{m-1})
	- \sum_{j=1,\,j\neq i}\frac{1}{N_j}\sum_{\ell=1}^{N_j}
	\E(\Delta\widetilde{K}_{ij}^\ell|\F_{m-1}). \nonumber
\end{align}
We perform a Taylor expansion of $K_{ij}$ at $\widetilde{X}_i^1(t_{m-1})
-\widetilde{X}_j^\ell(t_{m-1})$ and use the fact that
$K_{ij}$ is Lipschitz continuous with constant $L_{ij}$, such that $DK_{ij}$
can be bounded from above by $L_{ij}$:
$$
  \big|\E(\Delta\widetilde{K}_{ii}^\ell|\F_{m-1})\big|
	\le L_{ij}\big|\E(\Delta\widetilde{X}_{ij}^\ell|\F_{m-1})\big|
	+ \frac{d}{2}\|D^2K_{ij}\|_\infty\E\big(|\Delta\widetilde{X}_{ij}^\ell|^2
	\big|\F_{m-1}\big).
$$
Inserting
$$
  \Delta\widetilde{X}_{ij}^\ell 
	= \big(\widetilde{X}_i^1(t)-\widetilde{X}_i^1(t_{m-1})\big)
	+ \big(\widetilde{X}_j^\ell(t)-\widetilde{X}_j^\ell(t_{m-1})\big)
$$
into the previous estimate and taking into account the stability estimates
of Lemmas \ref{lem.stab1} and \ref{lem.stab2}, we infer that
\begin{align*}
  \big\|\E(\Delta\widetilde{K}_{ii}^\ell|\F_{m-1})\big\|
	&\le  C\tau L_{ij}\big(1+\theta^{\widetilde{q}_i}+\theta^{\widetilde{q}_j}\big) \\
	&\phantom{xx}{}+ C\tau\|D^2K_{ij}\|_\infty(1+\theta^{q'_i/2+1})
	\big(1+\| |\widetilde{X}_i^1(t_{m-1})|^{q'_i/2+1}\|\big) \\
	&\le C\tau\big(1+\theta^{q_i'+2}),
\end{align*}
where the constant $C>0$ does not depend
on $b_i$, $p_i$, or $m$. We use this estimate in \eqref{3.aux6} and observe that
$b_i/\min\{b_i,b_j\}\le\theta$, yielding
\begin{equation}\label{3.chi1}
  \big\|\E\big(\chi_i^1(\widetilde{X}_i^1(t))
	- \chi_i^1(\widetilde{X}(t_{m-1}))\big|\F_{m-1}\big)\big\|
	\le  C\tau\big(1+\theta^{q_i'+3}).
\end{equation}
Finally, we combine estimates \eqref{3.Zi1} and \eqref{3.chi1} to conclude from
\eqref{3.J21} that
$$
  J_{21} \le C\tau(1+\theta^{q'_i+3})\|Z_i^1(t)\| + C\tau^2(1+\theta^{3q'_i/2+3}).
$$

{\em Step 4: Estimate of $J_{22}$.} Set $\Delta K_{ij}^\ell
:=K_{ij}(\widetilde{X}_i^1(t)-\widetilde{X}_j^\ell(t))
- K_{ij}(X_i^1(t)-X_j^\ell(t))$. We use the Cauchy--Schwarz inequality and
\eqref{3.Z1} to obtain
\begin{align}\label{3.J22}
  J_{22} &\le \|Z_i^1(t)-Z_i^1(t_{m-1})\|\,\|\chi_i^1(\widetilde{X}(t))
	- \chi_i^1(X(t))\| \\
	&\le C\tau(1+\theta^{q'_i/2})\|\chi_i^1(\widetilde{X}(t)) - \chi_i^1(X(t))\| 
	\nonumber \\
	&\le C\tau(1+\theta^{q'_i/2})\bigg(\frac{1}{p_i-1}
	\bigg\|\sum_{\ell\in\C_{i,r},\,\ell\neq k}\Delta K_{ii}^\ell\bigg\|
	+ \frac{1}{N_i-1}\sum_{\ell=1,\,\ell\neq k}^{N_i}\|\Delta K_{ii}^\ell\| \nonumber \\
	&\phantom{xx}{}+ \sum_{j=1,\,j\neq i}^n\frac{b_i}{p_j\min\{b_i,b_j\}}\bigg\|
	\sum_{\ell\in\C_{j,r}}\Delta K_{ij}^\ell\bigg\|
	+ \sum_{j=1,\,j\neq i}^n\frac{1}{N_j}\sum_{\ell=1}^{N_j}\|\Delta K_{ij}^\ell\|
	\bigg). \nonumber 
\end{align}
The difference $\Delta K_{ij}^\ell$ can be estimated according to (see the
second inequality in \eqref{3.Z1})
$$
  |\Delta K_{ij}^\ell|\le L_{ij}\big(|Z_i^1(t)|+|Z_j^\ell(t)|\big)
	\le C\big(\tau\theta + |Z_i^1(t_{m-1})| + |Z_j^\ell(t_{m-1})|\big).
$$
Then, with the help of the auxiliary Lemma \ref{lem.S}, 
\begin{align*}
  \frac{1}{p_i-1}\bigg\|\sum_{\ell\in\C_{i,r},\,\ell\neq k}\Delta K_{ii}^\ell\bigg\|
	&\le \frac{C}{p_i-1}
	\bigg\|\sum_{\ell\in\C_{i,r},\,\ell\neq k}\big(\tau\theta + |Z_i^1(t_{m-1})| 
	+ |Z_i^\ell(t_{m-1})|\big)\bigg\| \\
	&\le C\tau\theta + C\|Z_i^1(t_{m-1})\|
	\le C\tau(1+\theta^{q'_i/2}) + C\|Z_i^1(t)\|, \\
	\frac{1}{p_j}\bigg\|\sum_{\ell\in\C_{j,r},\,\ell\neq k}\Delta K_{ij}^\ell\bigg\|
	&\le C\tau(1+\theta^{\gamma/2}) + C\|Z_i^1(t)\| + C\|Z_j^1(t)\|,
\end{align*}
where $\gamma=\max_{j=1,\ldots,n}q'_j$. Therefore, because of
$b_i/\min\{b_i,b_j\}\le\theta$, \eqref{3.J22} becomes
$$
  J_{22} \le C\tau(1+\theta^{q'_i/2})(1+\theta)\bigg(\tau(1+\theta^{\gamma/2})
	+ \sum_{j=1}^n\|Z_j^1(t)\|\bigg).
$$
We deduce from \eqref{3.J19J22} and the previous estimates for 
$J_{19},\ldots,J_{22}$ that
\begin{align}\label{3.chiZ}
  \sum_{i=1}^n\E\big(\chi_i^1(\widetilde{X}(t))\cdot Z_i^1\big)
	&\le C\tau^2(1+\theta^{3\gamma/2+3}) + C\tau(1+\theta^{\gamma+3})\sum_{i=1}^n
	\|Z_i^1(t)\| + C\tau\sum_{i=1}^n\Gamma_i \\
	&\le C_7(\theta)\tau^2 + C_8(\theta)\tau\bigg(\sum_{i=1}^n\|Z_i^1(t)\|^2\bigg)^{1/2}
	+ C_9\tau\sum_{i=1}^n\Gamma_i. \nonumber
\end{align}

{\em Step 5: End of the proof.} Let 
$$
  u(t)=\sum_{i=1}^n\|Z_i^1(t)\|^2, \quad
	r=\min_{i=1,\ldots,n}\bigg(r_i-2\sum_{j=1}^n\max\{L_{ij},L_{ji}\}\bigg) > 0.
$$
We infer from \eqref{3.dZi1dt} and \eqref{3.chiZ} that
$$
  \frac{\d u}{\d t} \le -ru + C_7\tau^2 + C_8\tau u^{1/2} 
	+ C_9\tau\sum_{i=1}^n\Gamma_i.
$$
The positive solution $z_+$ of the quadratic equation
$-rz^2 + C_7\tau^2 + C_8\tau z + C_9\tau\sum_{i=1}^n\Gamma_i = 0$
gives us an upper bound for $u(t)^{1/2}$, since $\d u/\d t\le 0$ otherwise.
Consequently,
\begin{align*}
  u(t)^{1/2} &\le z_+ = \frac{C_9}{2r}\tau + \frac{\sqrt{\tau}}{2r}
	\bigg(C_8^2\tau + 4C_7 r\tau + 4C_9 r\sum_{i=1}^n\Gamma_i\bigg)^{1/2} \\
	&\le C\tau(1+\theta^{3\gamma/2+3}) 
	+ C\sqrt{\tau}\bigg(\sum_{i=1}^n\Gamma_i\bigg)^{1/2}.
\end{align*}
This ends the proof of Theorem \ref{thm.main}.


\section{Particle systems with multiplicative noise}
\label{sec.multi}

The technique of the proof of Theorem \ref{thm.main} can be applied to
particle systems with multiplicative noise,
$$
  \d X_i^k = -\na V_i(X_i^k)\d t + \sum_{j=1}^n\alpha_{ij}\sum_{\substack{\ell=1 \\ 
	(i,k)\neq (j,\ell)}}^{N_j} K_{ij}(X_i^k-X_j^\ell)\d t 
	+ \sigma_i(X_i^k)\d B_i^k(t),
$$
with initial conditions \eqref{1.ic}, and $\alpha_{ij}=1/(N_j-\delta_{ij})$,
$i,j=1,\ldots,n$, $k=1,\ldots,N_i$. The random-batch process 
$\widetilde{X}_i^k$ is defined as in \eqref{1.rbm} but with 
$\sigma_i(\widetilde{X}_i^k)$ instead of $\sigma_i$.
In addition to Assumptions (A1)--(A4), we suppose the following conditions:

\renewcommand{\labelenumi}{(B\theenumi)}
\begin{enumerate}
\item Diffusion: $\sigma_i\in C^0(\R^d)$ is bounded and Lipschitz continuous with
Lipschitz constant $L_i>0$.
\item Strong convexity: The function $x\mapsto V_i(x)-r_{i}|x|^2/2$ is convex, where
$r_i>2\sum_{j=1}^n$ $\max\{L_{ij},L_{ji}\} + L_i^{2} d$ and 
$r_i>2L_i^2(2\max\{1,q_i\}+d-2)$, $i=1,\ldots,n$.
\end{enumerate}

\begin{theorem}\label{thm.main2}
Let Assumptions (A1)--(A2), (A4), (B1)--(B2) hold. Then there exists a constant
$C>0$, which is independent of $(b_i,p_i)_{i=1,\ldots,n}$, $m$, and $T$, such that
$$
  \sup_{0<t<T}\sum_{i=1}^n\|(X_i^k-\widetilde{X}_i^k)(t)\|
	\le C\sqrt{\tau}\bigg(1+\sum_{i=1}^n\Gamma_i\bigg)^{1/2}
	+ C\tau(1+\theta^\gamma),
$$
and $\theta$, $\gamma$, $\Gamma_i$ are defined in \eqref{1.theta}--\eqref{1.Gamma}.
\end{theorem}

\begin{proof}[Sketch of the proof]
The proof is similar to that one for Theorem \ref{thm.main} except for some
additional estimates for the multiplicative noise term. In particular,
Proposition \ref{prop.cons} keeps unchanged since it is concerned with the
shuffling process only. For the stability (Lemma \ref{lem.stab1}), we need
the condition $2\le q\le q'_i=2\max\{1,q_i\}$. 
The proof is essentially the same, except for the estimate
of the term $\frac12 q(q+d-2)\E(\sigma_i^2|X_i^k|^{q-2})$. Here, we use the
Lipschitz continuity of $\sigma_i$ and the stricter condition on $r_i$
in Assumption (B2).
In the estimate for $\widetilde{X}_i^k(t)-\widetilde{X}_i^k(t_{m-1})$
(Lemma \ref{lem.stab2}), the diffusion
$\sigma_i$ is controlled by the Lipschitz continuity,
$\sigma_i(\widetilde{X}_i^k)^2\le 2L_i^{2}|\widetilde{X}_i^k|^2 + 2\sigma_i(0)^2$,
and Lemma \ref{lem.stab1}.
Finally, for the control of the error process (Lemma \ref{lem.Z}), estimates 
\eqref{3.Z1}--\eqref{3.Z2} need to be changed to
\begin{align}
  & \|Z_i^k(t)-Z_i^k(t_{m-1})\| \le C\tau(1+\theta^{q'_i/2}) + C\sqrt{\theta}, 
	\label{5.est1} \\
	& \big|\E\big((Z_i^k(t)-Z_i^k(t_{m-1}))\chi_i^k(\widetilde{X}(t_{m-1}))\big)\big|  
	\le C\tau^2(1+\theta^{3q'_i/2}) 
	+ 8\tau\max_{j=1,\ldots,n}\|K_{ij}\|_\infty^2\Gamma_i \label{5.est2} \\
	&\phantom{xx}{}+ \sqrt{\tau}\bigg((1+\sqrt{\tau})(1+\theta^{q'_i})\|Z_i^k(t)\|
	+ \sqrt{\tau}(1+\sqrt{\tau})\sum_{j=1}^n\|Z_j^1(t)\|\bigg). \nonumber
\end{align}

For the proof of estimate \eqref{5.est1}, the right-hand side of \eqref{3.aux3} 
contains the additional term
$$
  \widetilde{J}_{14} = \bigg\|\int_{t_{m-1}}^t (\sigma_i(\widetilde{X}_i^k)
	- \sigma_i(X_i^k))\d B_i^k\bigg\|.
$$
The square of $\widetilde{J}_{14}$ is estimated by using the It\^o isometry 
and the Lipschitz continuity of $\sigma_i$. Integrating and taking the square root 
then leads to the additional $C\sqrt{\tau}$ term. 

The proof of \eqref{5.est2} is very similar to \eqref{3.Z2}, except that we need 
the inequality
$$
  \bigg\|\frac{1}{p_i-1}\sum_{\ell\in\C_{i,r},\,\ell\neq k}|Z_i^k|\bigg\|
	\le C\theta(\tau + \|Z_i^1(t_{m-1})\|).
$$
The square of the left-hand side is formulated as
$$
  \bigg\|\frac{1}{p_i-1}\sum_{\ell\in\C_{i,r},\,\ell\neq k}|Z_i^k|\bigg\|^2
	= \E\bigg\{\E\bigg[\bigg(\sum_{\ell\in\C_{i,r},\,\ell\neq k}|Z_i^k|\bigg)^2
	|\F_{m-1}\bigg]\bigg\}.
$$
Since $\xi_{m-1}$ is $\F_{m-1}$ measurable, the inner expectation becomes
\begin{align*}
  \E\bigg[\bigg(\sum_{\ell\in\C_{i,r},\,\ell\neq k}|Z_i^k|\bigg)^2
	|\F_{m-1}\bigg] &= \sum_{\ell,\ell'\in\C_{i,r},\,\ell,\ell'\neq k}
	\E\big(|Z_i^\ell|\,|Z_i^{\ell'}|\big|\F_{m-1}\big) \\
	&\le \sum_{\ell,\ell'\in\C_{i,r},\,\ell,\ell'\neq k}
	\sqrt{\E(|Z_i^\ell|^2|\F_{m-1})}\sqrt{\E(|Z_i^{\ell'}|^2|\F_{m-1})},
\end{align*}
using the Cauchy--Schwarz inequality for the conditional expectation.
A straightforward computation leads to
$$
  \E(|Z_i^\ell|^2|\F_{m-1}) \le C\theta^2(\tau + |Z_i^k(t_{m-1})|)^2,
$$
from which we infer that
$$
  \E\bigg[\bigg(\sum_{\ell\in\C_{i,r},\,\ell\neq k}|Z_i^k|\bigg)^2
	|\F_{m-1}\bigg]
	\le C\theta^2\bigg(\sum_{\ell\in\C_{i,r},\,\ell\neq k}(\tau + |Z_i^k(t_{m-1})|)
	\bigg)^2.
$$
As $Z_i^\ell(t_{m-1})$ is independent of $\xi_{m-1}$, 
the proof finishes after applying Lemma \ref{lem.S}.
\end{proof}

A more complicated particle system with multiplicative noise was considered in
\cite{CDHJ20}, which leads in a mean-field-type limit to the 
Shigesada--Kawasaki--Teramoto population model:
\begin{align*}
  & \d X_i^k = -\na V_i(X_i^k)\d t + \bigg(\sigma_i^2 
  + \sum_{j=1}^n f\bigg(\alpha_{ij}\sum_{\substack{\ell=1 \\ 
  (i,k)\neq(j,\ell)}}^{N_j}K_{ij}(X_i^k-X_j^\ell)\bigg)\bigg)^{1/2}\d B_i^k(t),
\end{align*}
with initial conditions \eqref{1.ic}, $i=1,\ldots,n$, $k=1,\ldots,N_i$, and 
the function $f$ is globally Lipschitz continuous. Again, the random-batch
process $\widetilde{X}_i^k$ is similar to \eqref{1.rbm}. 
For this system, we have been not able to prove an error
estimate of order $\sqrt{\tau}$, but only a stability estimate of the form
$$
  \sum_{i=1}^n\|(X_i^k-\widetilde{X}_i^k)(t)\|
  \le C(t)\sqrt{t}\bigg(\sqrt{\tau} h(t,\tau,\theta) + \sum_{i=1}^n\Gamma_i\bigg),
  \quad t>0,
$$
where $h(t,\tau,\theta)$ is a smooth function. Compared to the error estimates of
Theorems \ref{thm.main} and \ref{thm.main2}, 
the bound $\sum_{i=1}^n\Gamma_i$ for the
variance of the remainder \eqref{2.chi} is not multiplied by $\sqrt{\tau}$.
Numerical simulations
(not shown) reveal a saturation effect when $\tau$ becomes very small, indicating
that the previous estimate cannot be improved. 


\section{Numerical simulations}\label{sec.num}

We present numerical results for a test example, a population system, and
an opinion-formation model. The algorithm is implemented in Matlab.
The random shuffling is realized using the command {\tt randperm}, and
the stochastic differential equations are discretized by the standard
Euler--Maruyama scheme. 

\subsection{Discrete $L^2$ error for a test example}

We generalize the test example of \cite[Section 4.1]{JLL20}.
For this, we consider system \eqref{1.eq} with $n=3$ species in $d=2$ dimensions
and specify the functions
$$
  \na V_i(x) = r_i(x-m^{(i)}), \quad K_{ij}(x) = \frac{Q_iQ_j x}{1+|x|^2}, 
	\quad x\in\R^2,\ i,j=1,2,3,
$$
where the model parameters are $(Q_1,Q_2,Q_3)=(-1,2,-2)$, $(r_1,r_2,r_3)=(1,4,2)$,
and $m^{(1)}=(1,0)^T$, $m^{(2)}=-(1,1)^T$, $m^{(3)}=(1,1)^T$. This choice incorporates
different repulsive and attracting effects. The initial data are centered
Gaussian distributions with the variances $(v_1,v_2,v_3)=(2,2,1)$, 
where the index signifies the number of the species. 

For the first experiment, we choose the diffusion coefficients $\sigma_i=0.5$
for $i=1,2,3$ and the time step sizes $\tau=2^{-2},\ldots,2^{-6}$. The end time
is $T=1$, the batch sizes are $p_i=2$ for $i=1,2,3$, and the numbers $N_i$ of particles
of the $i$th species are $(N_1,N_2,N_3)=(100,100,200)$, $(1000,1000,2000)$, or
$(2500,2500,5000)$. Thus the total number of particles is $N=400$, $4000$, 
or $10000$.
We compare the random-batch solution with a reference solution, obtained by
solving the fully coupled system using the time step size $2^{-4},\ldots,2^{-8}$. 
Figure \ref{fig.convrate} (left) shows the discrete $L^2(\Omega)$ error for 
the different time step sizes, defined by
$$
  E = \bigg(\sum_{i=1}^n\frac{1}{N_i}\sum_{k=1}^{N_i}
	|\widetilde{X}_i^k(T)-X_i^k(T)|^2\bigg)^{1/2}.
$$
The reference line has the slope 1/2. The results clearly show
that the convergence rate is of order $O(\sqrt{\tau})$ as predicted by Theorem 
\ref{thm.main}. 

\begin{figure}[ht]
\includegraphics[width=80mm]{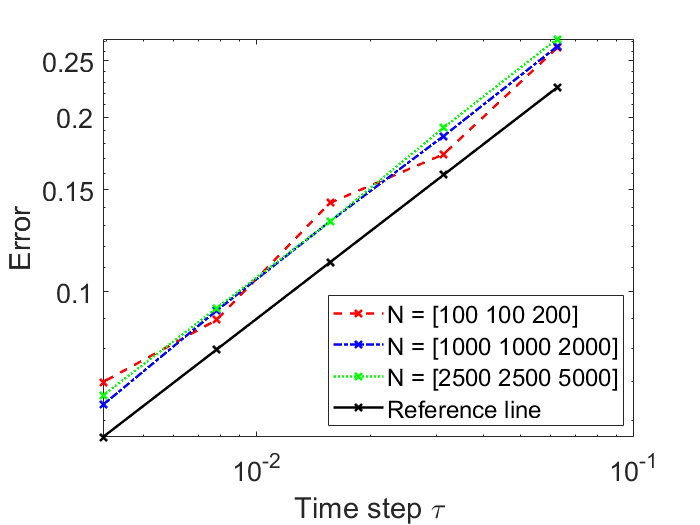}
\includegraphics[width=80mm]{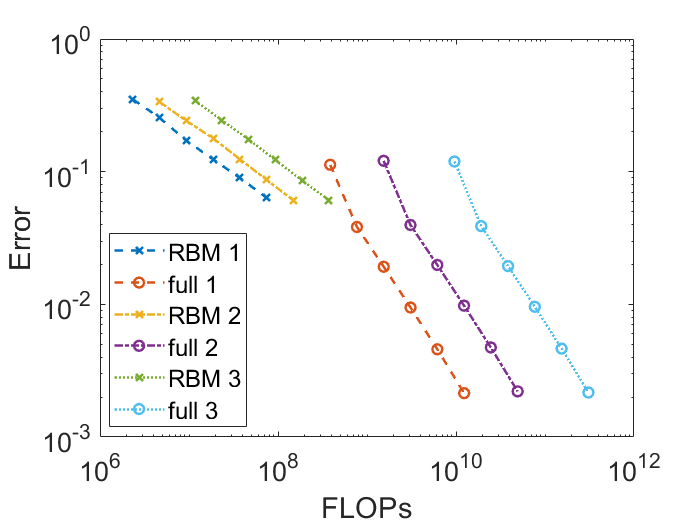}
\caption{Left: Discrete $L^2(\Omega)$ error $E$ versus time step size $\tau$
for various total particle numbers $N$.
Right: Discrete $L^2(\Omega)$ error versus number of FLOPs for various 
random-batch simulations (RBM) and the corresponding reference solutions (ref).}
\label{fig.convrate}
\end{figure}

Figure \ref{fig.convrate} (right) 
illustrates the $L^2(\Omega)$ error as a function of the
computational time, represented by the number of FLOPs (floating-point operations).
We choose $\sigma_i=0$ for all $i=1,\ldots,n$ to allow for the comparison of
the random-batch solution with a reference solution that is calculated beforehand.
The parameters for the random-batch algorithm are $T=1$, $n=2$, $d=2$, 
$(p_1,p_2)=(2,2)$,
$\tau=2^{-3},\ldots,2^{-7}$, and $(N_1,N_2)=(1250,1250)$ (RBM1, full 1),
$(2500,2500)$ (RBM2, full 2), or $(5000,5000)$ (RBM3, full 3). 
The reference solution is calculated from an explicit Euler scheme with the
time step size $\tau=2^{-1},\ldots,2^{-5}$. 
The number of FLOPs needed for the Matlab-internal functions are determined by the
lightspeed toolbox of Tom Minka (https://github.com/tminka/lightspeed). The total
numbers of FLOPs are then calculated by adding all needed operations manually.

Figure \ref{fig.convrate} (right) 
shows that the random-batch algorithm needs almost three
orders of magnitude less FLOPs than the reference algorithm. As expected, the discrete
$L^2(\Omega)$ error of the random-batch scheme is larger than that one of the
reference scheme for a given time step. 
However, for a given error, the number of FLOPs of the
random-batch algorithm is still much smaller compared to the reference algorithm,
namely by about two orders of magnitude.

\subsection{A population system}

We consider the population system derived in \cite{CDJ19} without external
potentials using the following parameters:
$n=3$, $d=1$, $T=2$, $N_i=5000$ for $i=1,2,3$, and
$(\sigma_1,\sigma_2,\sigma_3)=(1,2,3)$. The interaction kernels are given by
$K_{ij}=\na B_{ij}^\eta$, where $B_{ij}^\eta(x)=\eta^{-1}B_{ij}(x/\eta)$,
$B_{ij}(x)=D_{ij}\exp(1-1/(1-|x|^2))\mathrm{1}_{\{|x|<1\}}(x)$ for $x\in\R$, 
$\eta=2$, and
$$
  (D_{ij}) = \begin{pmatrix} 
  0 & 355 & 355 \\
  25 & 0 & 25 \\
  355 & 0 & 0 \end{pmatrix}.
$$
The initial data are Gaussian normal distributions with means
$(m_1,m_2,m_3)=(-1,2,3)$ and variances $(v_1,v_2,v_3)=(2,2,2)$.

Figure \ref{fig.popul} (left) illustrates the approximate probability densities
at time $T=2$ obtained by simulating the particle system 1000 times with 
the batch sizes $p_i=20$ for $i=1,2,3$ and the time step size $\tau=10^{-2}$.
We observe that the species segregate and avoid each other.
Each of the simulation requires about $2\cdot 10^{10}$ FLOPs, which needs to be
compared to about $5\cdot 10^{12}$ FLOPs required when using full interactions. 
This is a reduction of the numerical effort of more than two orders of magnitude.

Clearly, the reduction of computational cost comes at the price of an increased
error. Figure \ref{fig.popul} (right) presents the discrete $L^2(\Omega)$ error
versus the number of FLOPs for various configurations of the batch sizes
and various time step sizes. The end time is $T=1$, and we used batch sizes
$p_i=2,10,100,1000$ and time step sizes $\tau=2^{-1},\ldots,2^{-7}$.
The different points per line correspond to different values of $\tau$.
The reference solution is computed from the Euler--Maruyama scheme
with the step size $\tau=2^{-9}$; this simulation needed about $10^{13}$ FLOPs.
We see that the error decreases with the time step size and larger batch sizes.
The red dot in the figure indicates the number of FLOPs needed to
compute a numerical solution with full interactions and step size $\tau=10^{-2}$, 
to give a more practical point of reference. 
This simulation required about $2.5\cdot 10^{12}$ FLOPs,
while the random-batch algorithm with $\tau=2^{-7}$ was about four times faster.

\begin{figure}[ht]
\includegraphics[width=80mm]{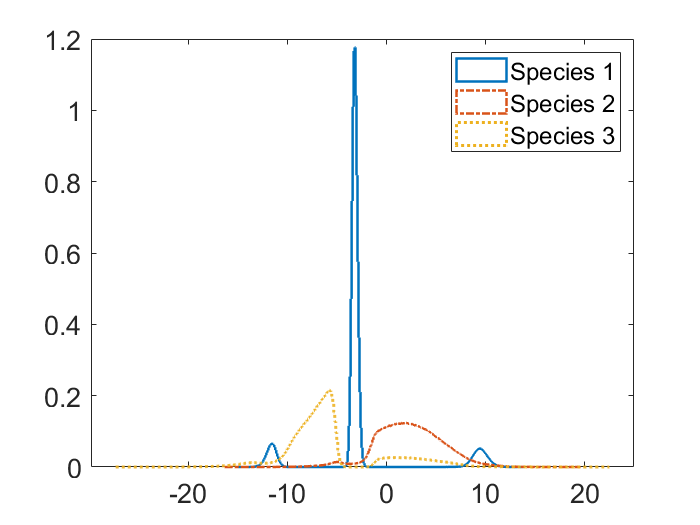}
\includegraphics[width=80mm]{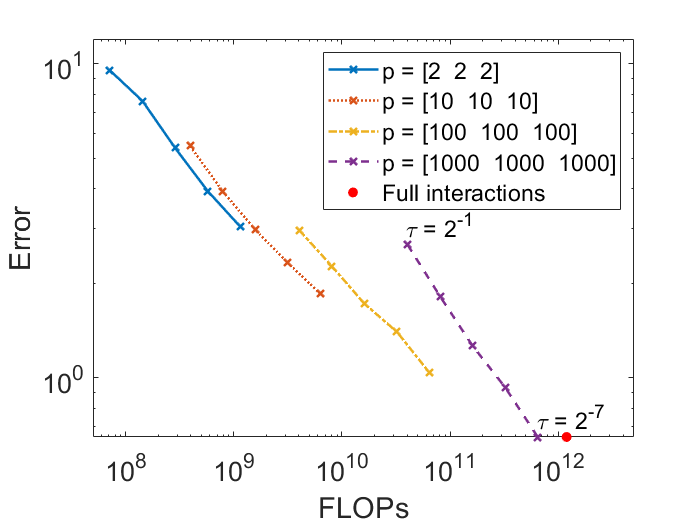}
\caption{Left: Histogram of the population model derived in \cite{CDJ19}
for three species at time $T=2$. Right: Discrete $L^2(\Omega)$ error versus
number of FLOPs for various batch sizes $p$ and time step sizes $\tau$. }
\label{fig.popul}
\end{figure}

\subsection{Opinion dynamics model}\label{sec.comp}

We model a company whose internal hierarchy regulates the communication
between three different types of agents: workers (species 1), managers (species 2),
and CEOs (species 3). The agents obey the following rules:
\begin{itemize}
\item CEOs can be only influenced by other CEOs. They influence managers 
(but not vice versa) and they do not interact with workers.
\item Managers can influence workers but not other managers or CEOs.
\item Workers can only influence each other.
\end{itemize}
The dynamics of opinions is described by the system
\begin{align*}
  & \d X_i^k(t) =  \sum_{j=1}^3\frac{1}{N_i-\delta_{ij}}
	\sum_{\ell=1,\,(i,k)\neq(j,\ell)}^{N_j}K_{ij}(X_i^k(t)-X_j^\ell(t))\d t
	+ \sigma\d t, \\
	& X_i^k(0) = X_{0,i}^k, \quad i,j=1,2,3, \ k=1,\ldots,N_i,\ 0<t\le T,
\end{align*}
which is a generalization of a model discussed in \cite{MoTa14}. 
The interaction is modeled by $K_{ij}(x)=-D_{ij}\phi(x/R_j)x$ for $x\in\R$,
where $\phi(x)=\exp(1-1/(1-|x|^{10}))\mathrm{1}_{(-1,1)}$ is a smooth approximation 
of the characteristic function $\mathrm{1}_{(-1,1)}$. 
The value $D_{ij}$ is a measure of 
the influence that an agent of species $j$ has over an agent of species $i$.
According to the above-mentioned interaction rules, the matrix $D=(D_{ij})$
has the structure  
$$
  D = \begin{pmatrix}
	D_{11} & D_{12} & 0 \\
	0      & 0      & D_{23} \\
	0      & 0      & D_{33}
	\end{pmatrix}.
$$

As the only way for CEOs to communicate with the workers happens indirectly
via the managers, we wish to explore the influence of the managers to achieve
a consensus. In particular, we consider managers that are very submissive to authority
($D_{23}\gg 1$) or that are less obedient ($D_{23}\le 1$). 
For the simulations, we use 5000 workers, 10 managers and 2 CEOs. The parameters
are $\sigma=0.1$, $T=5$, $\tau=10^{-5}$, and $(p_1,p_2,p_3)=(20,2,2)$.
The initial conditions are drawn from a uniform distribution on the interval
$[0,10]$. The interaction radii are $(R_1,R_2,R_3)=(1,2.5,5)$. 

In the first case (submissive managers), we choose the influence values
$$
  D_{11} = 5,\quad D_{12} = 10, \quad D_{23} = 25, \quad D_{33} = 0.1.
$$
Figure \ref{fig.comp1} (left) shows one simulation of the particle system. 
We observe that the managers are very eager to 
find a compromise between the opinions of the two CEOs. This change of the
opinion occurs too fast for the workers with more extreme opinions,
as they are not as susceptible as the managers (since $D_{12}<D_{23}$).
Therefore, they leave quickly the range of interaction of the managers and
form their own clusters. Only those workers who have an opinion already close
to that one of the CEOs, agree with the company policy and change their opinion
accordingly. 

\begin{figure}[ht]
\includegraphics[width=80mm]{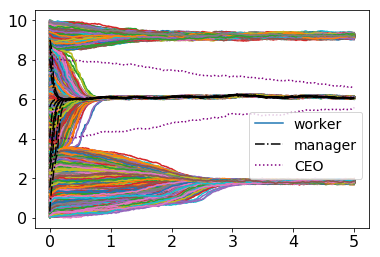}
\includegraphics[width=80mm]{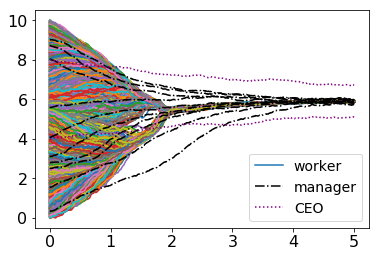}
\caption{Opinion versus time of the CEOs and managers in the case of very
submissive (left) or less obedient (right) managers.}
\label{fig.comp1}
\end{figure}

In the second case (less obedient managers), we choose the same values
of $D_{ij}$ as before except $D_{23}=1$. This means that the influence of the
CEOs over the managers is rather small. Figure \ref{fig.comp1} (right) shows that
the managers change their opinion slowly enough for the workers
to adapt their opinion, as they stay within their range of interaction. Eventually,
this leads to a consensus of opinion.

The simulations suggest that small changes over time are more likely to lead in 
an adjustment of the opinion and eventually to a consensus. In this picture, 
managers should not impose their opinion too quickly, but they should introduce
the changes sufficiently slowly such that the workers can adjust in time.

Finally, we explore the influence of the batch size on the running time and
the error. We consider 10000 workers, 100 managers, and 10 CEOs and choose
the parameters $\tau=2^{-3},\ldots,2^{-7}$, $T=4$, and $\sigma=0.1$.
The batch sizes are $(p_1,p_2,p_3)=(2,2,2)$, $(20,5,2)$, $(200,20,2)$, and
$(2000,20,2)$. Figure \ref{fig.batch} shows that the 
discrete $L^2(\Omega)$ error decreases
with larger batch sizes (since this involves more interactions), 
smaller time step sizes, or
$\theta$ closer to one, which is consistent with our error estimate. 
Clearly, the number of FLOPs increases with larger batch sizes.

\begin{figure}[ht]
\includegraphics[width=90mm]{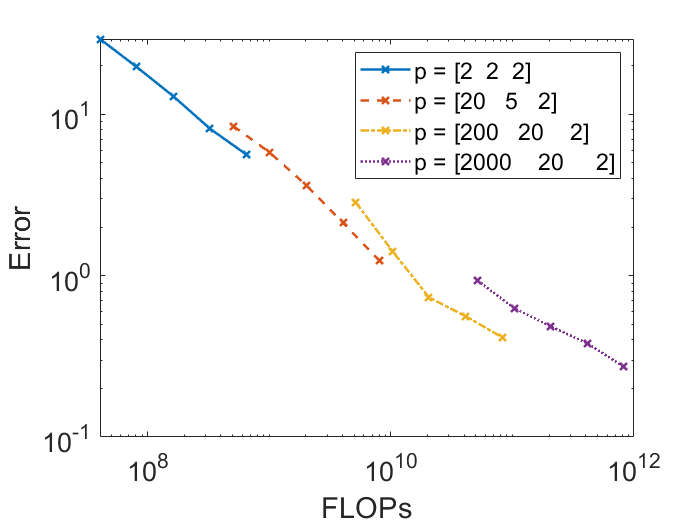}
\caption{$L^2(\Omega)$ error versus number of FLOPs for different batch sizes
and time step sizes $\tau=2^{-3},\ldots,2^{-7}$.}
\label{fig.batch}
\end{figure}

\begin{appendix}
\section{Auxiliary results}\label{app}

We recall some results involving the conditional expectation; see
\cite[Chapter 5]{Fel70}. Let $(\Omega,\F,\Prob)$ be a probability space.

\begin{lemma}\label{lem.cond0}
Let ${\mathcal H}$ be a sub-$\sigma$-algebra of ${\mathcal F}$ and let
$X$, $Y:\Omega\to\R^d$ be random variables such that $X$ is 
${\mathcal H}$-measurable. Then
$$
  \E(X|{\mathcal H}) = X, \quad \E(XY|{\mathcal H}) = X\E(Y|{\mathcal H}).
$$
In particular, the law of total expectation holds: $\E[\E(X|{\mathcal H})] = \E(X)$.
\end{lemma}

\begin{lemma}\label{lem.cond1}
Let $\G\subset\F$ be a $\sigma$-algebra, and $(X(t))_{t\ge 0}$ be an integrable 
stochastic process. Then, for any $t > 0$,
$$
  \E\bigg(\int_0^t X(s)\d s\bigg|\G\bigg) = \int_0^t\E(X(s)|\G)\d s.
$$
\end{lemma}

The lemma is a consequence of Fubini's theorem \cite[Lemma 2.3]{Bro72}.

\begin{lemma}\label{lem.cond2}
Let $T>0$, $(B(t))_{t\ge 0}$ be a $d$-dimensional Brownian motion, and 
$\F_t=\sigma(B(s),$ $s\le t)$ for $t\le T$. Furthermore, let $X(t)\in\R^d$ be a 
square integrable, progressively measurable process with respect to $\F_t$.
Then, for any $0\le s_1\le s_2\le T$,
$$
  \E\bigg(\int_{s_1}^{s_2}X(t)\d B(t)\bigg|\F_{s_1}\bigg) = 0.
$$
\end{lemma}

This lemma follows from the fact that $S(t):=\int_0^t X(s) \d B(s)$ is a 
martingale and consequently, $\E(S(s_1)-S(s_2))=0$ a.s.\ for $0\le s_1\le s_2\le T$.

\end{appendix}


\end{document}